\documentclass[runningheads]{llncs}
\usepackage[T1]{fontenc}
\usepackage{newtxtext}
\usepackage[varvw]{newtxmath}
\usepackage{graphicx}
\usepackage{color}
\usepackage{url}

\usepackage{mathtools,xspace}
\usepackage{xfrac}
\usepackage[dvipsnames,svgnames]{xcolor}
\usepackage{comment}
\usepackage{hyperref}
\usepackage{cleveref}
\usepackage{wrapfig}
\usepackage[font=small]{caption,subcaption}
\usepackage{todonotes}
\usepackage{etoolbox}
\usepackage{longtable}
\usepackage{tikz}
\usetikzlibrary{positioning}
\usetikzlibrary{arrows}
\usetikzlibrary{decorations.pathmorphing}
\usetikzlibrary{decorations.markings}
\usetikzlibrary{decorations.pathreplacing}
\usetikzlibrary{shapes.geometric}
\tikzset{
    small circles/.style={circle,inner sep=2pt,fill=#1},
    hollow circles/.style n args={2}{circle,inner sep=#1,draw=#2,thick},
    stars/.style={star,inner sep=2pt}
}
\tikzset{
    plain/.style={fill=none,shape=rectangle,draw=none,inner sep=3pt}
}
\newcommand{\colors}{\mathcal{C}}

\newcommand{\Maj}{\mathrm{Maj}}

\newcommand{\majcol}[1]{[#1]}
\newcommand{\configuration}[1]{$\textnormal{config}_{#1}$}
\newcommand{\move}[1]{\text{move} #1}
\newcommand{\half}[1]{\left \lfloor \frac{#1}{2} \right \rfloor}
\newcommand{\PSPACE}{\text{\normalfont PSPACE}\xspace}
\newcommand{\NP}{\text{\normalfont NP}\xspace}
\newcommand{\defproblem}[3]{
  \vspace{1mm}
\begin{center}
\noindent\fbox{
  \begin{minipage}{.9\linewidth}
  \begin{tabular*}{\linewidth}{@{\extracolsep{\fill}}lr}
    \textsc{#1} \
  \end{tabular*}
  {\bf{Input:}} #2  \
  \\ {\bf{Question:}} #3
  \end{minipage}
  }
\end{center}
  \vspace{1mm}
}
\hypersetup{
    colorlinks=true,
    linkcolor=blue,
    citecolor=blue,
    urlcolor=blue
}
\begin{document}
\title{Structural and computational aspects of majority coloring games}
%
%
\author{Yash Chawda\inst{1}\orcidID{0009-0003-3713-8556} \and
Saraswati Girish Nanoti\inst{2}\orcidID{0009-0009-7789-8895} \and
Brahadeesh Sankarnarayanan\inst{1}\orcidID{0000-0001-9191-1253} \and
Eshwar Srinivasan\inst{3}\orcidID{0000-0002-9233-1573}}
\authorrunning{Y. Chawda et al.}

\institute{
Department of Mathematics, Indian Institute of Technology Jodhpur,
Jodhpur 342030, Rajasthan, India\\
\email{p25ma0204@iitj.ac.in, brahadeesh@iitj.ac.in}\\
\url{https://yashu112.github.io}, \url{https://brahadeesh1994.github.io}
\and
Department of Computer Science and Automation,
Indian Institute of Science,
Bengaluru 560012, Karnataka, India\\
\email{saraswatig@iisc.ac.in}\\
\url{https://sites.google.com/iitgn.ac.in/saraswati-girish-nanoti/home}
\and
School of Computing,
Amrita Vishwa Vidyapeetham,
Coimbatore 641112, Tamil Nadu, India\\
\email{srireveshjay@zohomail.in, s\textunderscore eshwar@cb.amrita.edu}\\
\url{https://srireveshjay.github.io/eshwarsrinivasan.github.io}
}
\maketitle              
\begin{abstract}
This paper discusses two majority coloring games on graphs.
A \emph{majority coloring} of a graph \(G = (V,E)\) is a coloring of \(V(G)\) such that, for each vertex \(v\), the number of neighbors of \(v\) with the same color as \(v\) is at most \(\deg(v)/2\).
A \emph{strong majority coloring} is a coloring of \(V(G)\) such that, for each vertex \(v\), every monochromatic subset of \(N(v)\) has size at most \(\deg(v)/2\).
The \emph{(strong) majority coloring game} is a two player Maker-Breaker-type game, in which two players Alice and Bob color the vertices of a graph \(G\) alternately, maintaining the (strong) majority condition.
The least number of colors such that Alice has a winning strategy in such a game is called the \emph{(strong) majority game chromatic number} of the graph \(G\), denoted \(\mu_g(G)\) (or \(\mathrm{Maj}_g(G)\) for the strong version).

For the majority coloring game, we prove that \(\mu_g(G) \le 3\) under the following cases:
\(G\) is a \(2\)-caterpillar, \(G\) is a rooted tree with all leaves at depth \(k \le 4\), and \(G\) is a subdivision of some graph.
The latter resolves a problem posed by Bosek--Grytczuk--Jak\'obczak~\cite{BosekGrytczukEtAl2019}, who also asked whether \(\mu_g(T) \le 3\) for every tree \(T\).
For the latter question, we discuss various difficulties that arise when natural strategies are attempted by Alice to win the majority coloring game on trees.
We also include a comparison with the marking game and relaxed coloring game on trees, and with the majority coloring game on infinite locally finite acyclic graphs \(G\) with \(\delta(G) > 1\).

For the strong majority coloring game, we compute \(\mathrm{Maj}_g(C_n)\) exactly for each cycle \(C_n\), \(n \geq 3\).
We also initiate the study of the computational complexity of the strong majority coloring game; specifically, we prove that the decision version of the \textsc{Strong Majority Game Chromatic Number} problem is PSPACE-complete.
We also show that the \textsc{Strong Majority 2-Coloring} problem is NP-complete on Eulerian graphs.

\keywords{Majority coloring game \and Strong majority coloring game \and Caterpillar \and Subdivision \and Cycle \and PSPACE-complete \and NP-complete.}
\end{abstract}

\section{Introduction}
\label{S:Introduction}

\subsection{Coloring games on graphs}
Coloring games on graphs have been an active area of study since Gardner~\cite{Gardner1,Gardner2,Gardner3} popularized the map coloring game of Brams in his ``Mathematical Games'' column in \emph{Scientific American}.
Independently, Bodlaender~\cite{Bodlaender1991} in 1991 defined it over general graphs, calling it the coloring construction game.
The basic version---upon which several natural variations have been built---is a two-player Maker-Breaker-type game.
The two players, Alice and Bob, take turns to color a vertex of a graph \(G\) with a ``legal'' color from the available palette of colors: Alice's goal is to ensure every vertex of the graph is colored, and Bob's goal is to prevent this.
For instance, we may declare that a color is legal for a vertex \(v\) if that color has not been used on its neighborhood \(N(v)\); so, Alice wins if and only if she can ensure a proper coloring of \(G\) regardless of Bob's choice of play.
The least number of colors (i.e, the smallest palette size) with which Alice wins this coloring game on \(G\) is called the \emph{game chromatic number}, and denoted \(\chi_g(G)\).
By varying the coloring condition that determines the legal moves of Alice and Bob one obtains different coloring games, such as the greedy coloring game~\cite{CostaPessoaSampaioSoares2020}, connected coloring game~\cite{CharpentierHocquardSopenaZhu2020}, ordered vertex coloring game~\cite{Hollom2024}, and so on, as well as their corresponding game coloring parameters.

It was observed early on that the game version of a coloring parameter can have intricate behavior, even in regimes where the usual (static) version of the parameter is easily understood.
For instance, the game chromatic number is not monotone with respect to induced subgraphs; in fact, for each \(n \geq 3\) there is a graph \(G\) with \(\chi_g(G) = 3\) such that \(\chi_g(G - v) = n\) for some vertex \(v \in V(G)\) (cf.\ Dinski--Zhu~\cite{DinskiZhu1999}).
For the usual proper coloring of graphs, it is trivial that if \(G\) is \(k\)-colorable, then it is also \((k+1)\)-colorable.
However, Zhu's question~\cite{Zhu1999} whether the same holds for the coloring game is still open; that is, it is unknown whether a win for Alice with a palette of size \(k\) implies a win for Alice with a palette of size \(k+1\) as well, for the coloring game.
Zhu's question can be posed for each variation of the coloring game, and though it has been settled (in either direction) for a few variants, it remains an outstanding open question in the field of coloring games on graphs.
Computational questions have also been similarly resistant to easy answers.
For instance, Bodlaender's question~\cite{Bodlaender1991} from 1991 about the hardness of the coloring game was settled by the PSPACE-completeness results of Costa--Pessoa--Sampaio--Soares~\cite{CostaPessoaSampaioSoares2020} only as recently as 2020.
Moreover, the complexity is still unknown when the number of colors is not a part of the input to the problem.

\subsection{Majority colorings as a local variant of defective colorings}
In this paper, we study the game version of the following coloring problem, known as a \emph{majority coloring} or an \emph{unfriendly partition} of a graph.
Given a simple graph \(G = (V,E)\) and a set \(C \subset \mathbb{N}\) of colors, a (possibly improper) vertex coloring \(c \colon V(G) \to C\) of \(G\) is a majority coloring if, for each \(v \in V(G)\), \(\lvert\{ w \in N(v) : c(w) = c(v) \}\rvert \leq \deg(v)/2\).
The minimum number of colors on which \(G\) has a majority coloring is its \emph{majority chromatic number}, denoted \(\mu(G)\).
The majority chromatic number itself has certain unusual properties, which deserve a few comments before we consider the game version.

Firstly, by a result of Lov\'asz~\cite{Lovasz1966} from 1966, every nonempty finite graph has \(\mu(G)=2\).
Cowen--Emerson~\cite{CowenEmerson1985,ShelahMilner1990} in 1985 conjectured that all nonempty infinite graphs also have a majority \(2\)-coloring, which came to be known as the \emph{Unfriendly Partition Conjecture}.
This conjecture has been well studied, and partial results are known for line graphs, uncountable graphs, locally finite infinite graphs and infinite graphs which have finite vertices of infinite degree~\cite{AharoniMilnerEtAl1990,Berger2017,BruhnDiestelEtAl2010,PekalaPrzybylo2025,ShelahMilner1990}, among other classes.
In particular, Shelah--Milner~\cite{ShelahMilner1990} in 1990 showed that the conjecture is false for uncountably infinite graphs, but it remains open in general for countably infinite graphs.
Moreover, every graph is known to be majority \(3\)-colorable~\cite{ShelahMilner1990}, and similar results on the list version~\cite{OubornyPitz2026} and online version~\cite{AnholcerBosekEtAl2025} has also recently been proved.

Majority colorings can be viewed as a local variant of \emph{defective colorings}, which were studied by Cowen--Cowen--Woodall~\cite{CowenCowenWoodall1986} in 1986.
In a \(d\)-defective \(r\)-coloring of a graph \(G\), at most \(d\) neighbors of any given vertex can have the same color as itself.
Thus, where the number of neighbors of a vertex permitted to have the same color as itself is proportional to its degree in a majority coloring, this bound is globally defined in a defective coloring, and hence fixed for each vertex.
We will have more to say on the relationship between the game versions of majority and defective colorings in later sections.

Another variation of majority colorings was recently introduced by Kalinowski--Kamyczura--Pil{\'s}niak--Wo{\'z}niak~\cite{KalinowskiEtal2026} in 2026: a \emph{strong majority coloring} of a graph \(G\) is a (possibly improper) vertex coloring of \(G\) such that, for each vertex \(v\), any monochromatic subset of its neighborhood has size at most \(\deg(v)/2\).
The least number of colors on which \(G\) has a strong majority coloring is called its \emph{strong majority chromatic number}, denoted \(\mathrm{Maj}(G)\).
Note that a strong majority coloring is only defined for graphs \(G\) with \(\delta(G) \geq 2\).
In~\cite{KalinowskiEtal2026} it is proved that unlike the majority chromatic number, \(\mathrm{Maj}(G)\) can be unbounded over all finite graphs \(G\).
They also prove general bounds on \(\mathrm{Maj}(G)\) in terms of \(\Delta(G)\), and compute the parameter exactly for cycles and complete graphs, as follows:
\begin{theorem}[{\cite[Observations 5 and 6, Theorem 10]{KalinowskiEtal2026}}]\label{T:strong_maj_results}
    \hfill
    \begin{enumerate}
        \item\label{T:strong_maj_results_cycles} \(\Maj(C_n) = 2\) if \(n \equiv 0 \pmod{4}\), and \(\Maj(C_n) = 3\) otherwise.
        \item\label{T:strong_maj_results_kn} \(\Maj(K_n)=3\) if \(n\neq4\), and \(\Maj(K_4) = 4\).
        \item\label{T:strong_maj_results_max_deg} If \(\delta(G) \geq 2\), then \(\Maj(G)\le 2\Delta(G)+1\).
        Moreover, for infinitely many integers \(\Delta\), there exist graphs \(G\) such that \(\Delta(G)=\Delta\) and \(\Maj(G)=2\Delta+1\).
    \end{enumerate}
\end{theorem}

\subsubsection{Related coloring notions.}
We shall focus on the game versions of majority colorings and strong colorings in this work.
In particular, these are vertex colorings, and we note that majority edge colorings (the static version) have also been studied in the literature~\cite{BockKalinowskiEtAl2023,KalinowskiPilsniakStawiski2025,PekalaPrzybylo2025}.

We also briefly survey several variants of improper colorings which are similar in spirit to majority colorings but which behave quite differently, and are therefore not the focus of this work.
\begin{itemize}
    \item Hind--Molly--Reed~\cite{HindMollyReed1997} in 1997 introduced \emph{r-frugal colorings} of graphs.
    A coloring of a graph is called \(r\)-frugal if no vertex has more than \(r\)-members of any color class in its neighborhood.
    The \(r\)-frugal chromatic number of a graph \(G\), sometimes denoted \(\chi^f_r(G)\)~\cite{BresarHuSamadi2026}, is the least number of colors in an \(r\)-frugal coloring of \(G\).
    \item \emph{Alliances} in graphs were introduced by Kristiansen--Hedetniemi--Hedetniemi~\cite{KristiansenHedetniemiHedetniemi2004} in 2004.
    One such notion is called a \emph{defensive alliance} in a graph \(G\), which is a set of vertices \(S\subseteq V(G)\) having the property that \(\forall v\in S\), \(\vert N[v]\cap S \vert \ge \vert N(v) \cap (V-S) \vert\).
    \item Motivated by unfriendly partitions and related notions, \emph{cost effective domination} in graphs was studied by Tabitha L. McCoy in 2012~\cite{McCoy2012}.
    A vertex \(v\) in a dominating set \(D\) is called cost effective if it is adjacent to at least as many vertices in \(V \setminus D\) as it is in \(D\).
    If all vertices in a dominating set \(D\) are cost effective, then \(D\) is called a cost effective dominating set and the minimum cardinality of such a set is called the \emph{cost effective domination number} of the graph, denoted \(\gamma_{c\epsilon}(G)\).
	\item Bujt{\'a}s--Sampathkumar--Tuza--Pushpalatha--Vasundhara~\cite{BujtasEtal2011} in 2011 introduced the \emph{k-improper C-colorings} of graphs and defined the \emph{k-improper upper chromatic number} \(\bar\chi_{k\text{-imp}}(G)\) of a graph \(G\) as the \emph{maximum} number of colors that can be used to color \(G\) such that for any \(v\in V(G)\), at most \(k\) vertices in neighborhood of \(v\) are colored differently than \(v\).
	\item Motivated by majority colorings, Bujt{\'a}s--Dettlaff--Furma{\'n}czyk--Laskowska~\cite{BujtasDettlaffEtal2026} in 2026 introduced the \emph{Majority C-coloring} of a graph \(G\), in which at least half the neighbors of any given vertex should have the same color as that vertex.
    The maximum number of colors that permit a majority C-coloring of \(G\) is called the majority C-chromatic number, denoted \(\bar\chi_{\ge}(G)\).
	Bujt{\'a}s--Dettlaff--Furma{\'n}czyk--Laskowska--Tuza~\cite{BujtasDettlaffEtal2026_2} in 2026 then studied the majority C-colorings in cartesian product graphs.
	\item Freyberg--Marr~\cite{FreybergMarr2024} in 2024 introduced \emph{neighborhood balanced colorings}:
	it is a \(2\)-coloring of the vertices of a graph \(G\) such that each vertex has equal number of neighbors of each color.
	Such a coloring partitions \(V(G)\) into two sets such that each vertex has an equal number of neighbors in each set.
	\item Almeida--Gupta--Pawar--Singh~\cite{AlmeidaSinghEtal2025} in 2025 introduced \emph{neighborhood balanced \(k\)-colorings}, generalizing the notion of neighborhood balanced colorings: it is a coloring of a graph \(G\) using \(k\) colors such that each vertex has an equal number of neighbors of each color.
    \item Chellali--Hedetniemi--Meddah~\cite{ChellaliHedetniemiMeddah2025} in 2025 studied \emph{minority sets} in graphs which are defined as follows: A set of vertices, \(M \subseteq V\) of a graph \(G=(V,E)\) is called a minority set if each vertex \(v\in M\) has strictly more neighbors in \(V\setminus M\) than it has in \(M\).
    Note that an \emph{unfriendly partition} (or a 2-majority coloring) of a graph \(G\) gives two minority sets \(M_1\) and \(M_2\) such that \(M_1 \cup M_2 = V(G)\) and \(M_1 \cap M_2 = \phi \).
    \item Chellali--Hedetniemi--Meddah~\cite{ChellaliHedetniemiMeddah2026} in 2026 studied \emph{majority sets} in graphs.
    A subset of vertices \(S\subseteq V\) of a graph \(G\) is called a majority set if every vertex \(v\in S\) has more neighbors in \(S\) than it has in \(V \setminus S\).
    
\end{itemize}

\subsection{Majority and relaxed (defective) coloring games: known results}

\subsubsection{The majority coloring game.} The game version of majority colorings was introduced by Bosek--Grytczuk--Jak{\'o}bczak~\cite{BosekGrytczukEtAl2019} in 2019.
In the \emph{majority coloring game} both Alice and Bob are required to maintain the majority condition at each stage; that is, at every partial coloring as the game progresses, it should hold that for every vertex \(v\), at most half the neighbors of \(v\) have the same color as \(v\). (Of course, if \(v\) is uncolored, then this condition is automatically satisfied at \(v\).)
The minimum number of colors on which Alice has a winning strategy, that is, she can get a majority coloring of \(G\) regardless of Bob's strategy, is called the \emph{majority game chromatic number} of \(G\), and is denoted \(\mu_g(G)\).

It was shown in~\cite{BosekGrytczukEtAl2019} that \(\mu_g(G)\) is bounded above by the \emph{game coloring number} of \(G\), \(\mathrm{col}_g(G)\), which is defined via the \emph{marking game}:
the marking game on a graph \(G\) of order \(n\) is a two-player game where the players alternately choose vertices of \(G\), thereby inducing a linear order on \(V(G)\), say \(v_1 \prec v_2 \prec \dotsb \prec v_n\).
With respect to this ordering, the set of \emph{back neighbors} of \(v_i\) is
\(
B(v_i) := \{u \in N(v_i) : u \prec v_i \}
\).
The \emph{back degree} of \(v_i\) is the cardinality of the set \(B(v_i)\).
The back degree of \(G\) is defined as
\(
B(G) := \max_{1\le i\le n}\{B(v_i)\}
\).
The goal of Alice is to minimize \(B(G)\) whereas the goal of Bob is to maximize it.
The game coloring number \(\mathrm{col}_g(G)\) is the least \(k\) such that Alice has a strategy for the marking game on \(G\) such that \(B(G) < k\).
\begin{theorem}[{\cite[Theorem 2]{BosekGrytczukEtAl2019}}]
	For any graph \(G\), \(\mu_g(G) \leq \mathrm{col}_g(G)\).
\end{theorem}
Essentially, if Alice follows the marking game strategy in the majority coloring game, then fewer than \(\mathrm{col}_g(G)\) colors are forbidden for the vertex that she chooses in each round, so the bound follows.
On the other hand, \(\mu_g\) is unbounded on the class of bipartite graphs and the class of \(2\)-degenerate graphs~\cite[Theorems 1 and 5]{BosekGrytczukEtAl2019}.
The authors also prove the following result for subdivisions of a graph:
\begin{theorem}[{\cite[Theorem 4]{BosekGrytczukEtAl2019}}]
	Let \(G\) be a graph and let \(G'\) be the \(1\)-subdivision of \(G\).
	Then, \(\mu_g(G') \leq 3\).
\end{theorem}
Bosek--Grytczuk--Jak\'obczak also remark that a similar result can be proved for \(2\)-subdivisions of a graph, and that the general case is not clear.

In a recent work~\cite{ChawdaNanotiSankarnarayanan2026}, the first three authors of this article proved bounds on \(\mu_g\) for certain classes of graphs:
\begin{theorem}[{\cite[Theorem 3.7, Section 4]{ChawdaNanotiSankarnarayanan2026}}]\label{T:FSTTCS}
	\hfill
	\begin{enumerate}
		\item\label{T:FSTTCSa} If \(G\) is a tree with \(\Delta(G) \leq 4\), then \(\mu_g(G) \leq 3\).
		\item\label{T:FSTTCSb} If \(G\) is a nonempty path, star, or complete graph, then \(\mu_g(G) = 2\).
	\end{enumerate}
\end{theorem}
Faigle--Kern--Kierstead--Trotter~\cite{FaigleKernEtAl1993} showed that \(\mathrm{col}_g(T) \leq 4\) for any tree \(T\), so Theorem~\ref{T:FSTTCS}(\ref{T:FSTTCSa}) improves this bound for a certain class of trees.
Whether the bound can be improved for general trees is an open question raised in~\cite{BosekGrytczukEtAl2019}.

Since \(\mu(G) = 2\) for every finite graph, computing the majority chromatic number is a constant-time problem.
However, the pre-coloring extension problem for majority colorings can be much more complicated.
This was first introduced in~\cite{ChawdaNanotiSankarnarayanan2026}, where the authors proved hardness results were proved for the static and game versions of the pre-coloring extension problem:
\begin{theorem}[{\cite[Theorem 5.2, 5.4, 6.1, 6.5]{ChawdaNanotiSankarnarayanan2026}}]\label{T:FSTTCS_complexity}
	\hfill
    \begin{enumerate}
        \item\label{T:FSTTCS_complexity_1} The \textsc{Extended Majority Coloring} problem is \NP-complete, even when $k=3$ and the input graph is planar with degree at most $12$.
        \item\label{T:FSTTCS_complexity_2} The \textsc{Extended Majority $2$-Coloring} problem is \NP-complete, even on bipartite graphs. 
        \item\label{T:FSTTCS_complexity_3} The \textsc{Extended Majority Game Chromatic Number $(G,n,k)$} is \PSPACE-complete.
        \item\label{T:FSTTCS_complexity_4} The \textsc{Extended Majority $2$-Coloring Game} is \PSPACE-complete, even for bipartite graphs.
    \end{enumerate}
\end{theorem}

\subsubsection{The relaxed coloring game.} In contrast with the majority coloring game, the game version of defective colorings, also known as the \emph{\(d\)-relaxed \(r\)-coloring game}, was introduced by Chou--Wang--Zhu~\cite{ChouWangZhu2003} in 2003 and has since been heavily studied.
In the \(d\)-relaxed \(r\)-coloring game played on a graph \(G\), Alice and Bob play on a palette of \(r\) colors, maintaining the \(d\)-defective coloring condition at each stage.
The smallest \(r\) which guarantees a winning strategy for Alice is called the \emph{\(d\)-relaxed game chromatic number} of \(G\), denoted \(\chi_g^d(G)\).
Several results regarding the relaxed game chromatic number of trees~\cite{ChouWangZhu2003,HeWuZhu2004}, outerplanar graphs~\cite{HeWuZhu2004,DunnKierstead2004,WuZhu2006,WuZhu2008}, graphs with cut-vertices~\cite{Sidorowicz2015} and multipartite graphs~\cite{Dunn2012} have been established.
We will draw on the following results of Chou--Wang--Zhu~\cite{ChouWangZhu2003} and He--Wu--Zhu~\cite{HeWuZhu2004} to prove similar results for the majority coloring game in Section~\ref{S:Infinite_Acyclic}.
\begin{theorem}[{\cite{ChouWangZhu2003,HeWuZhu2004}}]\label{T:defective-game-trees}
	For a nonempty tree \(T\), \(\chi_g^1(T) \leq 3\) and \(\chi_g^d(T) = 2\) for \(d \geq 2\).
\end{theorem}

\subsection{Our results}

Our first result resolves a problem raised by Bosek--Grytczuk--Jak\'obczak~\cite{BosekGrytczukEtAl2019} on finding the game majority chromatic number of an arbitrary subdivision of a graph.
\begin{theorem}\label{T:Main_subdivisions}
    Let \(G\) be a finite simple graph, and let \(s \colon E(G) \to \{ 1, 2, 3, \dotsc\}\) be any function from the edges of \(G\) to the positive integers.
    Let \(G'\) be obtained from \(G\) by making \(s(e)\) many subdivisions of each edge \(e \in E(G)\).
    Then \(\mu_g(G') \le 3\).
\end{theorem}

Next, we improve on the upper bound \(\mu_g(G) \leq 4\) coming from the game coloring number for more classes of trees:
\begin{theorem}\label{T:Main_fixed_depth}
    For nonempty rooted trees \(T\) in which all leaves have depth \(k\), where \(k \leq 4\), we have \(\mu_g(T) \le 3\).
\end{theorem}
\begin{theorem}\label{T:Main_2caterpillars}
    For any \(2\)-caterpillar \(C\), \(\mu_g(C)\le 3\).
\end{theorem}
Since there are caterpillars \(C\) with \(\mathrm{col}_g(C) = 4\) (cf.~\cite{Bodlaender1991}), this is a genuine improvement.
These results also support the idea that \(\mu_g(T) \leq 3\) for every tree \(T\), which was raised as an open question by Bosek--Grytczuk--Jak\'obczak~\cite{BosekGrytczukEtAl2019}.

We also undertake a detailed investigation of why the latter problem seems to be difficult to resolve in Section~\ref{S:Natural_strategy}.
We provide several natural strategies for Alice using a palette of \(3\) colors, and construct specific trees on which Bob defeats these strategies.
We also show why the strategies conceived for the marking game by Faigle--Kern--Kierstead--Trotter~\cite{FaigleKernEtAl1993}, and for the relaxed coloring game by Chou--Wang--Zhu~\cite{ChouWangZhu2003} and He--Wu--Zhu~\cite{HeWuZhu2004} do not directly help on the majority coloring game.
These negative results motivate some open problems, whose resolution can further progress towards a complete understanding of the game majority chromatic number of trees.
However, we observe that the strategies proving Theorem~\ref{T:defective-game-trees} extend to the majority game coloring on certain classes of infinite acyclic graphs, therefore we can say that if \(G\) is an infinite acyclic graph with \(2m \le \deg(v) \le 2m+1 \hspace{0.2cm} \forall v \in V(G) \) where \(m\ge 2\), then, \(\mu_g(G) \le 2\).

In the later part of the paper, we introduce the \emph{strong majority coloring game}, in which Alice and Bob are required to maintain the strong majority condition at each partial stage of the game.
The \emph{strong majority game chromatic number}, \(\mathrm{Maj}_g(G)\), of the graph \(G\) is the least number of colors with which Alice has a winning strategy in the strong majority coloring game on \(G\).
We study the strong majority game chromatic number of cycles and prove the following result.
\begin{theorem}\label{T:Main_cycles}
    For cycles \(C_n\) of length \(n \geq 3\), \(\Maj(C_n) = 3\) if \(n \neq 4\) and \(\mathrm{Maj}_g(C_n)=2\) otherwise.
\end{theorem}

Lastly we investigate the complexity of the strong majority game coloring and show:
\begin{theorem}\label{T:Main_pspace_complete}
The decision version of the \textsc{Strong Majority Game Chromatic Number} is \PSPACE-complete (when $k\geq 3$).
\end{theorem}

\begin{theorem}\label{T:Main_np_complete}
    The \textsc{Strong Majority $2$-Coloring} problem is \NP-complete on Eulerian graphs.
\end{theorem}

\section{Preliminaries}
\label{S:Preliminaries}

Before we sketch the proofs of the main results, we describe some key aspects of the majority coloring game which will dictate the strategies and counterstrategies of Alice and Bob throughout this paper.

Call a vertex \(v\) to be \emph{majority colored} if it has \(\lfloor \deg(v)/2 \rfloor\) neighbors that have the same color as itself.
If \(v\) is majority colored with color \(i\), we shall denote this as \majcol{i}.
Firstly, since both Alice and Bob are required to maintain the majority condition at each stage of the game, a winning condition for Bob is the existence of an uncolored vertex for which no legal color is available (for either player).
Assuming the game is played on a palette of three colors, such a ``critical'' vertex \(v_c\) can arise under the following conditions:
\begin{enumerate}
	\item \(v_c\) has three majority colored neighbors, each of a different color (cf.\ Figure~\ref{F:config1});
	\item \(v_c\) has two majority colored neighbors of different colors, and it has \(\lfloor \deg(v_c)/2 \rfloor+1\) neighbors of the third color (cf.\ Figure~\ref{F:config2}).
\end{enumerate}
A winning strategy for Alice with \(3\) colors must prevent the occurence of either configuration at any point in the game.
In the case of Theorem~\ref{T:FSTTCS}(\ref{T:FSTTCSa}), the constraint on the maximum degree automatically prevents the formation of the second configuration.
Hence, a winning strategy for Alice on such trees only needs to prevent the formation of the first configuration.
This is achieved in the following way: the tree is rooted at the first vertex that Alice colors, so the critical configurations to be avoided are those in Figure~\ref{F:crit_config_trees}. Then, broadly speaking, in subsequent moves she ensures that there does not appear any majority colored vertex with an uncolored parent.
It is shown in~\cite{ChawdaNanotiSankarnarayanan2026} that this can be done and is indeed a winning strategy for Alice.

This philosophy carries over to the results presented in this work.
We find that certain structural constraints---edge subdivisions, low depth, small paths attached to a central spine (i.e., \(2\)-caterpillars)---permit Alice to focus on tackling only one of the two critical configurations.
Furthermore, we show how natural extensions of such strategies to general trees can fail.
This suggests that if indeed \(\mu_g(T) \leq 3\) for all trees \(T\), then this must be due to deeper reasons.

 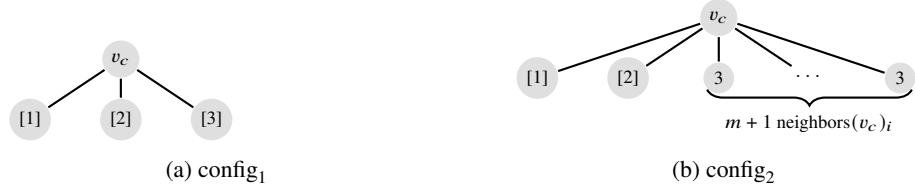
\begin{figure}[h!]
    \begin{subfigure}{0.45\textwidth}
        \begin{tikzpicture}
            [level distance=8mm,
            font=\scriptsize,
            every node/.style={fill=gray!25,circle, inner sep=2pt},
            level 1/.style={sibling distance=12mm},
            line width=0.8pt,
            smooth]

            \node {$v_c$}
            child{node{\majcol{1}}}
            child{node{\majcol{2}}}
            child{node{\majcol{3}}};
        \end{tikzpicture}
        \caption{\configuration{1}}
        \label{F:config1}
    \end{subfigure}
    \hfill 
    \begin{subfigure}{0.45\textwidth}
    \begin{tikzpicture}
        [every node/.style={fill=gray!25,circle, inner sep=2pt},
        level distance=8mm,
        font=\scriptsize,
        sibling distance=12mm,
        line width=0.8pt,
        smooth]

        \node (root) {$v_c$}
        child { node (c1) {\majcol{1}} }
        child { node (c2) {\majcol{2}} }
        child { node (c3) {3}}
        child { node (c4) [plain]{$\cdots$} }
        child { node (c5) {3} };

        \draw [decorate,decoration={brace,mirror,amplitude=6pt}]
        (c3.south west) -- (c5.south east)
        node [plain] [midway,below=6pt] {\(m+1\) $\text{neighbors}(v_c)_i$};
        \end{tikzpicture}
        \caption{\configuration{2}}
        \label{F:config2}
    \end{subfigure}
    \caption{Critical configurations for the majority coloring game.}
    \label{crit_config}
    \end{figure}
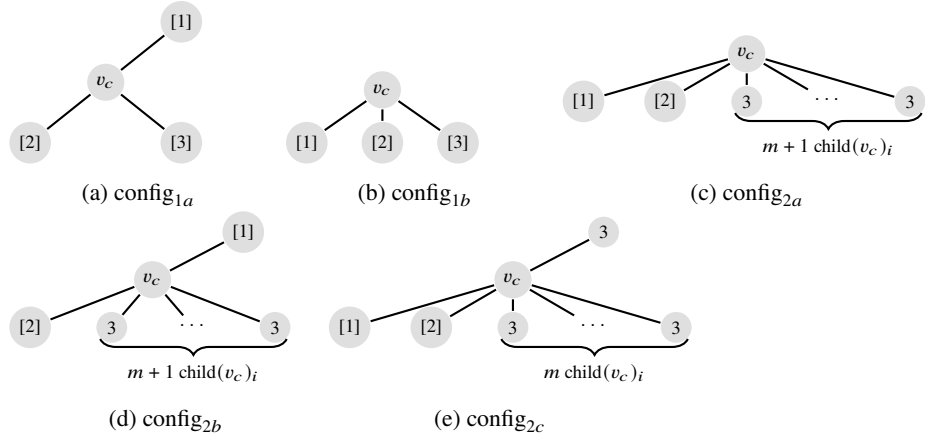
\begin{figure}[h!]
    \begin{subfigure}[b]{0.28\textwidth}
        \begin{tikzpicture}
            [level distance=8mm,
            font=\scriptsize,
            every node/.style={fill=gray!25,circle, inner sep=2pt},
            level 2/.style={sibling distance=20mm},
            level 3/.style={sibling distance=10mm},
            line width=0.8pt,
            smooth]

            \node {\majcol{1}}
            child{node [xshift=-10mm] {$v_c$}
            child{node {\majcol{2}}}
            child{node {\majcol{3}}}
            };
        \end{tikzpicture}
        \caption{\configuration{1a}}
        \label{config1a}
    \end{subfigure}
    \begin{subfigure}[b]{0.28\textwidth}
        \begin{tikzpicture}
            [level distance=7mm,
            font=\scriptsize,
            every node/.style={fill=gray!25,circle, inner sep=2pt},
            level 1/.style={sibling distance=10mm},
            line width=0.8pt,
            smooth]

            \node {$v_c$}
            child{node{\majcol{1}}}
            child{node{\majcol{2}}}
            child{node{\majcol{3}}};
        \end{tikzpicture}
        \caption{\configuration{1b}}
        \label{config1b}
    \end{subfigure}
    \begin{subfigure}[b]{0.4\textwidth}
    \begin{tikzpicture}
        [every node/.style={fill=gray!25,circle, inner sep=2pt},
        level distance=7mm,
        font=\scriptsize,
        sibling distance=12mm,
        line width=0.8pt,
        scale=0.9,
        smooth]

        \node (root) {$v_c$}
        child { node (c1) {\majcol{1}} }
        child { node (c2) {\majcol{2}} }
        child { node (c3) {3}}
        child { node (c4) [plain]{$\cdots$} }
        child { node (c5) {3} };

        \draw [decorate,decoration={brace,mirror,amplitude=6pt}]
        (c3.south west) -- (c5.south east)
        node [plain] [midway,below=6pt] {\(m+1\) $\text{child}(v_c)_i$};
        \end{tikzpicture}
        \caption{\configuration{2a}}
        \label{config2a}
    \end{subfigure}
    \begin{subfigure}[b]{0.34\textwidth}
    \begin{tikzpicture}
        [every node/.style={fill=gray!25,circle, inner sep=2pt},
        level distance=7mm,
        font=\scriptsize,
        sibling distance=12mm,
        line width=0.8pt,
        scale=0.9,
        smooth]

        \node {\majcol{1}}
        child { node [xshift=-12mm] {$v_c$}
        child { node (c1) {\majcol{2}} }
        child { node (c2) {3}}
        child { node (c3) [plain]{$\cdots$} }
        child { node (c4) {3} }
        };

        \draw [decorate,decoration={brace,mirror,amplitude=6pt}]
        (c2.south west) -- (c4.south east)
        node [plain] [midway,below=6pt] {\(m+1\) $\text{child}(v_c)_i$};

        \end{tikzpicture}
        \caption{\configuration{2b}}
        \label{config2b}
    \end{subfigure}
    \begin{subfigure}[b]{0.35\textwidth}
    \begin{tikzpicture}
        [every node/.style={fill=gray!25,circle, inner sep=2pt},
        level distance=7mm,
        font=\scriptsize,
        sibling distance=12mm,
        line width=0.8pt,
        scale=0.9,
        smooth]

        \node {3}
        child { node [xshift=-12mm] {$v_c$}
        child { node (c1) {\majcol{1}} }
        child { node (c2) {\majcol{2}} }
        child { node (c3) {3}}
        child { node (c4) [plain]{$\cdots$} }
        child { node (c5) {3} }
        };

        \draw [decorate,decoration={brace,mirror,amplitude=6pt}]
        (c3.south west) -- (c5.south east)
        node [plain] [midway,below=6pt] {\(m\) $\text{child}(v_c)_i$};

        \end{tikzpicture}
        \caption{\configuration{2c}}
        \label{config2c}
    \end{subfigure}
    \caption{Critical configurations on a rooted tree to be avoided by Alice.}
    \label{F:crit_config_trees}
\end{figure}

\section{Subdivided graphs are game majority \(3\)-colorable}
\label{S:Subdivisions}

Let \(G\) be a finite simple graph.
Let \(s \colon E(G) \to \{1,2,3,\dotsc \}\) be a function which assigns a positive integer to each edge of \(G\).
Let \(G'\) be obtained from \(G\) by making \(s(e)\)-many subdivisions of the edge \(e\), for each \(e \in E(G)\).
We claim that \(\mu_g(G') \le 3\).

The key ideas are as follows.
Subdividing an edge creates vertices of degree \(2\) which can never be critical.
Moreover, since every edge is subdivided at least once, the vertices of degree more than \(2\)---those that are ``potentially critical''---form an independent set, i.e., they are far away from each other.
Alice can exploit this by ensuring that she colors old vertices with high priority, and she has enough room to color a neighbor of on old vertex with a safe color if Bob tries to color old vertices before she does.

More formally, Alice takes the following strategy:
    \begin{enumerate}
        \item Alice colors any old vertex on her first move.
        
        \item Suppose Bob colors a new vertex \(v\) with an uncolored old vertex \(w\) as a neighbor.
        If more than one such vertex \(w\) exists, pick one arbitrarily.
        Then, Alice colors \(w\) with a color different from that on \(v\), and minimally used in \(N(w)\).
        If more than one such color exists, pick one arbitrarily.
    
        \item If Bob colors a new vertex with no uncolored old neighbor, or if Bob colors an old vertex, then:
    
        \begin{enumerate}
            \item Alice colors a neighbor \(x\) of a colored old vertex, with a color different from both neighbors of \(x\).
            \item If no such vertex exists, Alice colors any uncolored old vertex with a color minimally used in its neighborhood.
            \item If no such old vertex exists, then Alice colors any new vertex with a color different from both its neighbors.
        \end{enumerate}
    \end{enumerate}

    Now we show that the above strategy is a winning strategy for Alice.

\begin{lemma}
    Suppose that Alice does not create an uncolored old vertex with a majority colored neighbor. Then at any time there can be at most one uncolored old vertex with a majority colored neighbor, which Alice colors in the very next move.    
\end{lemma}
\begin{proof}
    Take the first instance of a counterexample.
    Then we must have one of the following cases:
    \begin{itemize}
        \item There is a \(2\)-subdivided edge in which both ends are uncolored, and both internal vertices are given the same color; i.e., there is a path \(uvwx\) where \(u\) and \(x\) are old and uncolored, and \(v\) and \(w\) are new and colored the same.
        \item There are two \(1\)-subdivided edges in which both ends are uncolored, and both internal vertices and the common old vertex they are adjacent to are all colored the same; i.e., there is a path \(uvwxy\), where \(u,w,y\) are old, \(v,x\) are new, and \(v,w,x\) are all colored the same.
    \end{itemize}
    Consider the first scenario.
    We ask, who colored \(v\) and \(w\)?
    By Alice's strategy, she would not have colored \(v\) or \(w\) while \(u\) and \(x\) are still uncolored, so Bob must have colored them.
    But as soon as Bob colors the first one, Alice would have colored its old neighbor.
    
    Next, consider the second scenario.
    Again, by Alice's strategy, she would have colored \(v\) or \(x\) only after \(w\) was colored, and in that case she would have chosen a color different from that on \(w\).
    Thus, Bob must have colored \(v\) and \(x\), but then Alice would have colored the old neighbor of whichever vertex between these that Bob colored first.
    
    Thus, neither configuration can arise.
\end{proof}

\begin{lemma}
    If \(v\) is an uncolored old vertex with a majority colored neighbor \(w\), then Bob created this configuration and Alice colors \(v\) in the very next move.
\end{lemma}
\begin{proof}
    Consider the first instance of a counterexample.
    By her strategy, Alice colors \(w\) only if \(x\), the other neighbor of \(w\), is an old colored vertex, but then she chooses a color for \(w\) different from that on \(x\), so \(w\) won't be majority-colored: such a choice of color is always possible since \(w\) has degree 2.
    So, the situation must have been that Bob colored \(w\), and then Alice colored \(x\), the other (old) neighbor of \(w\).
    As per strategy, she would have chosen a color for \(x\) different from that on \(w\), thereby avoiding making \(w\) majority-colored.
    Why is such a choice possible? If not, then it means the two remaining colors must be forbidden at \(x\), which is only possible if \(x\) has at least one majority-colored neighbor. 
    By minimality, this configuration was created by Bob, which meant that \(x\) would have been colored by Alice in the very next move after its creation, so \(x\) would have been colored before \(w\).
    The only case that \(x\) would not have been colored immediately in the next move is that there is another uncolored old vertex adjacent to a majority-colored vertex.
    But this is not possible by the previous lemma.
\end{proof}

\begin{theorem}
    Let \(G'\) be obtained from a finite simple graph \(G\) by subdividing each edge \(e \in E(G)\) \(s(e)\)-many times, where \(s(e) \ge 1\) is a positive integer.
    Then \(\mu_g(G') \le 3\).
\end{theorem}
    
\begin{proof}
    If Bob wins, he wins at an old vertex because every edge is subdivided at least once (so there is always at least one color available on a new vertex).
    Suppose \(v\) is the old vertex with no legal color available.
    Then it must have at least two majority-colored neighbors \(x\), and \(y\), with colors 1 and 2, say. Who colored \(x\) and y?
    As per the strategy Alice would have colored them only if they were neighbors of some colored old vertices.
    However, in this case, her choice of color would have avoided making them majority-colored, and such a choice always exists because only at most one color needs to be avoided on them.
    So, Bob colored both \(x\) and \(y\), and suppose he colored \(x\) before y.
    Why did Alice not color \(v\) immediately after Bob colored x?
    It could only be because \(x\) has another uncolored old vertex as a neighbor, which Alice chose to color.
    But then Alice would have chosen a color different from 1 for that neighbor, so that \(x\) is not majority-colored.
    Such a choice would not have been possible for Alice only if the remaining two colors were already forbidden.
    But then at least one of the two colors is forbidden because of a majority-colored neighbor, and the previous lemma says Alice would have colored it promptly.
    Hence the theorem.
\end{proof}

We remark that the bound in Theorem~\ref{T:Main_subdivisions} is tight even when the parent \(G\) is chosen among the class of graphs for which \(\mu_g(G) = 2\).
For example, let \(G\) be the claw graph \(K_{1,3}\), and \(G'\) be its \(1\)-subdivision.
Then \(\mu_g(G) = 2\), but \(\mu_g(G') = 3\).
Remarkably, if \(G''\) is the \(2\)-subdivision of \(G\), then we again have \(\mu_g(G'') = 2\) as we show below in Section~\ref{S:1subdivision_claw}.
Moreover, it is possible to show that if \(\tilde{G}\) is a large enough subdivision of \(G\), then we have \(\mu_g(\tilde{G}) = 3\) again, as shown below in Section~\ref{S:2subdivision_claw}.
This highlights the unusual behavior of the parameter under taking edge subdivisions.

\begin{figure}[h]
        \centering
        \begin{subfigure}{0.3\textwidth}
        \centering
            \begin{tikzpicture}[
                every node/.style={circle, draw, fill=white, inner sep=2pt},
                line width=0.9pt,
                smooth
                ]
        
                \node (v0) at ( 0,1) {};
                \node (v1) at ( 0,0) {};
                \node (v2) at (-1,0) {};
                \node (v3) at ( 1,0) {};

                \draw (v2)--(v0)--(v1);
                \draw (v0)--(v3);
            \end{tikzpicture}
            \caption{\(\mu_g(G)=2\)}
            \label{F:claw_a}
        \end{subfigure}
        \begin{subfigure}{0.3\textwidth}
        \centering
            \begin{tikzpicture}[
                every node/.style={circle, draw, fill=white, inner sep=2pt},
                line width=0.9pt,
                smooth
                ]
        
                \node (v0) at   ( 0,1) {};
                \node (v1) at (-.5,.5) {};
                \node (v2) at  (0, .5) {};
                \node (v3) at ( .5,.5) {};
                \node (v4) at  ( -1,0) {};
                \node (v5) at    (0,0) {};
                \node (v6) at    (1,0) {};

                \draw (v5)--(v2)--(v0)--(v1)--(v4);
                \draw (v0)--(v3)--(v6);
            \end{tikzpicture}
            \caption{\(\mu_g(G')=3\)}
            \label{F:claw_b}
        \end{subfigure}
        \begin{subfigure}{0.3\textwidth}
        \centering
            \begin{tikzpicture}[
                every node/.style={circle, draw, fill=white, inner sep=2pt},
                line width=0.9pt,
                smooth
                ]
        
                \node (v0) at (0,1.5) {};
                \node (v1) at (-.5,1) {};
                \node (v2) at ( 0,1) {};
                \node (v3) at (.5,1) {};
                \node (v4) at (-1,.5) {};
                \node (v5) at ( 0,.5) {};
                \node (v6) at ( 1,.5) {};
                \node (v7) at ( -1,0) {};
                \node (v8) at ( 0,0) {};
                \node (v9) at ( 1,0) {};

                \draw (v8)--(v5)--(v2)--(v0)--(v1)--(v4)--(v7);
                \draw (v0)--(v3)--(v6)--(v9);
            \end{tikzpicture}
            \caption{\(\mu_g(G'')=2\)}
            \label{F:claw_c}
        \end{subfigure}
        \caption{Value of \(\mu_g(G)\) for \(G = K_{1,3}\) and its subdivisions}
        \label{F:claw}
    \end{figure}
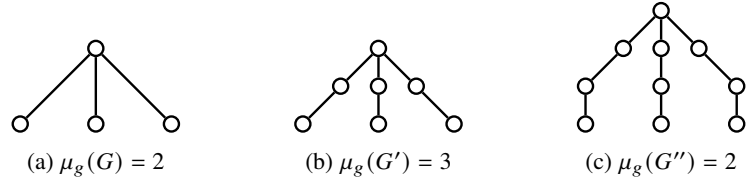

\subsection{1-Subdivision of Claw}
\label{S:1subdivision_claw}
Now we show that the bound is tight.
First note that Alice just needs two colors to win on \(K_{1,3}\) (as it is equivalent to a 3-star).
Now suppose Alice and Bob play the majority coloring game using 2 colors on the tree shown in figure~\ref{F:tree_depth2a}.
We will show that no matter how Alice plays, she is unable to avoid either of the partial colorings shown in figure~\ref{F:tree_depth2b}, where the vertex 6 is uncolored and the gray vertices may or may not be colored.

\begin{figure}
        \centering
        \begin{subfigure}{0.45\textwidth}
        \centering
            \begin{tikzpicture}
                [level distance=6mm,
                every node/.style={draw, circle, inner sep=1pt},
                level 1/.style={sibling distance=12mm},
                line width=0.9pt,
                smooth]
    
                \node {1}
               child{node{2}
                    child{node {3}}
                    }
                child{node{4}
                    child{node {5}}
                    }
                child{node{6}
                    child{node {7}}
                    };
            \end{tikzpicture}
            \caption{}
            \label{F:tree_depth2a}
        \end{subfigure}
        \begin{subfigure}{0.45\textwidth}
        \centering
            \begin{tikzpicture}
                [level distance=6mm,
                every node/.style={draw, circle, inner sep=1pt},
                level 1/.style={sibling distance=12mm},
                line width=0.9pt,
                smooth]
    
                \node [fill=red!30]{1}
                child{node [fill=red!30]{2}
                    child{node [fill=gray!20]{3}}
                    }
                child{node [fill=gray!20]{4}
                    child{node [fill=gray!20]{5}}
                    }
                child{node{6}
                    child{node [fill=cyan!30]{7}}
                    };
            \end{tikzpicture}
            \caption{}
            \label{F:tree_depth2b}
        \end{subfigure}
        \caption{Tree with all leaves at depth 2 and its unavoidable configurations in a majority coloring game with 2 colors}
        \label{F:tree_depth2_tight}
    \end{figure}
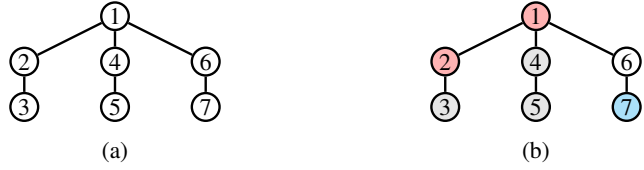

Suppose Alice starts by coloring 1.
Without loss of generality we may assume she starts with red, as we can swap the role of colors if she starts with blue.
Bob then colors 3 with blue which forces a red on 2.
If Alice colors 2 with red then Bob colors 7 with blue which forces a red on 6.
Note that this is the configuration shown in figure~\ref{F:tree_depth2b}.
But degree of 1 is 3 which means at most \(\half{3}=1\) of its neighbors can have the same color as itself.
Since 2 is already red, Alice is left with no legal color for 6.
Suppose Alice doesn't color 2.
Clearly, 4-5 and 6-7 are equivalent due to the symmetry and hence without loss of generality, it is sufficient to consider the cases of Alice coloring 4 and 5.
But no matter what she colors on 4 or 5, Bob can color 6 with red and we get the reflected version of~\ref{F:tree_depth2b}.
Again, Alice loses as this time she is left with no legal color for 2.

Suppose Alice doesn't start by coloring 1 and instead colors 2 (equivalent to 4 and 6 due to symmetry).
Bob then colors 1 with the same color and Alice again loses as before.
If Alice starts by coloring a leaf, say 7 (equivalent to 3 and 5 by symmetry), then Bob colors the 1 with the opposite color and again Alice loses.

Hence 3 colors are strictly needed for this tree which shows that the bound is tight.

We shall revisit the 1-subdivision of claw in Section~\ref{S:Fixed_depth_trees} (cf. figure~\ref{F:tree_depth2_tight}).

\subsection{2-Subdivision of Claw}
\label{S:2subdivision_claw}
For now we consider the 2-subdivision of claw shown in figure~\ref{F:claw_c}.
We draw it again in a more convenient manner with its vertices labeled.
We will show that Alice can win on this graph with two colors.

\begin{figure}[h]
    \centering
    \begin{tikzpicture}[
            every node/.style={circle, draw, fill=white, inner sep=1pt},
            line width=0.9pt,
            smooth
        ]

        \node (0) at (0,0) {0};
        \node (1) at (1,0) [rectangle, inner sep=2pt]{1};
        \node (2) at (2,0) {2};
        \node (3) at (3,0) [rectangle, inner sep=2pt]{3};
        \node (4) at (4,0) {4};
        \node (9) at (5,-1) {9};
        \node (5) at (5,0) [rectangle, inner sep=2pt]{5};
        \node (6) at (6,0) {6};
        \node (7) at (3,-1) {7};
        \node (8) at (4,-1) [rectangle, inner sep=2pt]{8};

        \draw (0)--(1)--(2)--(3)--(4)--(5)--(6);
        \draw (3)--(7)--(8)--(9);
    \end{tikzpicture}
    \caption{2-subdivision of \(K_{1,3}\)}
    \label{F:2-subdiv_star}
\end{figure}
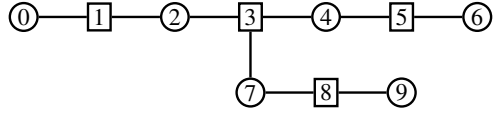

First we make some observations regarding the structure of the graph and its implications on the majority coloring of the graph.
The vertices 0,6,9 are pendants and as a result a legal color is always available for them, so Alice need not worry about them.
Since 0 is a pendant vertex, 0 and 1 cannot have same color and as a result 1 cannot be majority colored.
Hence, at 2, at most one color can be forbidden, which is when 3 and 4 are colored same.
The same is true for vertices 4 and 7 due to symmetry.
Therefore, Alice only needs to make sure that she is able to color 1,3,5,8, which have been drawn differently to highlight the same.
These are precisely the vertices which can have two majority colored neighbors, if Alice is able to color them, she wins the game.

Alice starts by coloring 1 with red.
Bob can play his first move in the following ways:
\begin{itemize}
    \item If Bob colors 3, Alice colors Alice colors 5 with the same color.
    Now whatever Bob colors, Alice just colors 8 with a legal color.
    She can do this because in one move Bob can forbid at most one color at 8, either by coloring 9 or by coloring 7 with same color as 3.
    Hence she wins if Bob colors 3.
    \item If Bob colors 5 or 8 instead of 3, Alice colors 3 and she again wins as before.
    \item If Bob colors 0, Alice colors 3 with the same color as 1.
    Now if Bob colors 2, Alice colors 5 same as 3 and wins as before and if he colors 4 (or 7), Alice colors 5 (or 8) with a legal color and again wins.
    \item If Bob colors 2, Alice colors 3 with opposite color.
    Again, if Bob colors 4 (or 7), Alice colors 5 (or 8) and wins as before.
    And if he colors 6 (or 9), then also Alice wins by coloring 5 (or 8).
    \item If Bob starts by coloring 4, Alice colors 3 with the opposite color.
    Now at most one color can be forbidden at 5 and it has a legal color available.
    Alice thus colors 8 in the next move, regardless of what Bob colors and wins.
    Due to symmetrical structure, similar argument holds if Bob starts by coloring 7.
    \item If Bob colors 6, Alice colors 5 with the opposite color.
    Now if Bob colors 2,4 or 7, Alice colors 3 and if he colors 9, Alice colors 8.
    She again wins as before.
\end{itemize}

Thus no matter how Bob plays, Alice wins the game.
Hence for 2-subdivision of a \(K_{1,3}\), \(\mu_g(G)=2\).
This example again shows a very unusual behavior of the parameter \(\mu_g(G)\).
It first increased when number of vertices increased from \(K_{1,3}\) to its 1-subdivision and then decreased from its 1-subdivision to its 2-subdivision.

\section{Trees with all leaves at depth \(k \leq 4\) are game majority \(3\)-colorable}
\label{S:Fixed_depth_trees}

Let \(T\) be a nonempty rooted tree in which all leaves are at depth \(k\), where \(k \in \{1,2,3,4\}\).
If \(k = 1\), then \(T\) is a star, so \(\mu_g(T) = 2\) by Theorem~\ref{T:FSTTCS}(\ref{T:FSTTCSb}).
So let us assume \(k\in \{2,3,4\}\).

The key idea here is that if all leaves are at the same depth \(k\), then Alice can ensure that a critical vertex cannot appear on a vertex of depth \(k - 1\) by immediately coloring the parent once Bob colors a leaf.
If the overall depth of the tree is small, then critical vertices can also be avoided near the root.
Thus, a tree of ``small'' depth will simply have no wiggle room for Bob to create a critical vertex.

We begin with a lemma that is important for the proofs.

\begin{lemma}\label{L:not_maj_n-1}
    In any tree of depth at least 2, if the parent of a vertex at level \(n-1\) is not colored, the vertex cannot be majority colored. (root represents the $0^{th}$ level and \(n\) is the deepest level)
\end{lemma}
\begin{proof}
    Consider a vertex \(v\) at level \(n-1\) of a tree whose parent(v) at level \(n-2\) is uncolored.
    Since \(n \ge 2\), \(v\) is not the root. \(v\) has one parent and the rest of its neighbors are leaves.
    At most one of its neighbors, precisely parent(v), can have the same color as \(v\).
    But parent(v) is uncolored by assumption.
    All children of \(v\) are leaves.
    If any of them has the same color as \(v\), then it would break the majority condition for that leaf.
    Thus no neighbor of \(v\) has the same color as \(v\).
    But for \(v\) to be majority colored, we need exactly $m=\half{\deg(v)}$ of its neighbors to be of the same color.
    Thus \(v\) does not get majority colored.

    \emph{Note:} If degree of \(v\) is 2 or 3 and its parent has the same color then \(v\) can be majority colored.
\end{proof}

\begin{theorem}\label{T:depth_2}
    For trees in which all leaves have depth 2, \(\mu_g(T) \le 3\).
\end{theorem}
\begin{proof}
    In her first move, Alice colors the root.
    For leaves, only one color may be forbidden, the color of its parent (if it is colored).
    So levels 0 and 2 are taken care of, we only need to define a strategy for level 1 which ensures that 3 colors are not forbidden for a vertex at level 1.
    
    If Bob colors a leaf and its parent is uncolored, color the parent with a color different from the leaf and the root.
    If Bob colors a leaf and its parent is colored, color a vertex at level 1 with a color different from the root.
    If such a vertex is not available, color any leaf.
    If Bob colors a vertex at level 1, color another vertex at level 1, if available, else color any leaf.

    This strategy ensures that she never needs more than 3 colors and hence \(\mu_g(T) \le 3\).

    Moreover, the bound is tight.
    Recall the 1-subdivision of claw discussed in Section~\ref{S:1subdivision_claw}.
    Note that it is identical to a tree with all leaves at depth 2 and it has \(\mu_g(G)=3\) which proves that the bound is tight.
\end{proof}

\begin{theorem}\label{T:depth_3}
    For trees in which all leaves have depth 3, \(\mu_g(T) \le 3\).
\end{theorem}
\begin{proof}
    In her first move, Alice colors the root.
    For leaves, only one color may be forbidden, the color of its parent (if it is colored).
    So levels 0 and 3 are taken care of, we only need to define a strategy for levels 1 and 2.
    If Bob colors a vertex at level 2 or 3 and its parent is uncolored, Alice colors the parent using a color minimally used in its neighbors.
    Otherwise, she colors any other uncolored vertex at level 1, with a color different from the root, which is least used in its children.
    
    If such a vertex is not present, she colors vertices at level 2.
    She does not color a leaf unless all other internal nodes are colored.
    
    For any critical configuration at a vertex \(v\), we need at least one majority colored child of uncolored \(v\).
    Thus, any critical configuration cannot occur at level \(1\), by Lemma~\ref{L:not_maj_n-1}.

    The critical configurations cannot occur at level 2 also.
    For any critical configuration at a vertex \(v\), we need at least 2 children of uncolored \(v\) to be colored.
    Alice does not color a leaf unless all internal nodes are colored by our strategy.
    If Bob colors any leaf and its parent is uncolored, Alice colors it in the next move.

    Thus, 3 colors are sufficient for trees of depth 3.
\end{proof}

\begin{theorem}\label{T:depth_4}
    For trees in which all leaves have depth 4, \(\mu_g(T)\le 3\).
\end{theorem}
\begin{proof}
    In her first move Alice colors the root.
    If in any move Bob colors a vertex which leads to the creation of majority colored vertex, Alice in her move colors the parent of the majority colored vertex, if it is uncolored.
    If Bob colors a vertex at level 2, Alice colors its parent at level 1.
    Otherwise, she colors any uncolored vertex at level 1, with a color different from the root.
    If no such vertex is available, she colors a vertex at level 2, 3 or 4, whichever is available first.

    Let's see how this strategy ensures 3 colors are sufficient.
    If Bob colors a leaf, it becomes majority colored, and Alice would color the parent if it is uncolored.
    Thus, no critical configuration arises at level 3.
    The critical configurations cannot arise at leaves and the root was colored in the first move.
    Also, by Lemma~\ref{L:not_maj_n-1}, there can be no critical configuration at level 2.
    Thus, level 0, 2, 3, 4 are taken care of.
    
    We claim that our strategy ensures that a critical configuration does not arise at level 1.
    First, we show that following our strategy, Alice does not create a majority colored vertex with uncolored parent at level 1.
    There are 2 ways she could have done this:
    \begin{itemize}
        \item She colored a vertex \(v\) at level 2 with a color, say \(c\), and $m=\half{\deg(v)}$ of its children were already of color \(c\).
        But by our strategy, she wouldn't have chosen it over the uncolored vertex at level 1 unless there was a majority colored child of \(v\).
        But the children of \(v\) are at level 3 and since \(v\) is not colored yet, by Lemma~\ref{L:not_maj_n-1}, none of its child at level 3 can be majority colored.
        Hence, this situation does not arise.
        
        \item She colored a vertex at level 3 with \(c\) and \(v\) at level 2 and $m-1$ of its children were already \(c\).
        She would have chosen to color this vertex only if it had a majority colored child.
        But it's children are leaves.
        Let's say Bob colored that leaf \(c'\)
        Now, Alice could chose a color different from \(c\) and \(c'\) and she doesn't create a majority colored vertex.
    \end{itemize}

    Consequently, none of the configurations 1a, 1b, 2a and 2c shown in Figure~\ref{F:crit_config_trees} can occur at a  vertex level 1.
    Because, as soon as Bob creates one of the two majority colored vertices required for the configuration, Alice colors the parent at level 1 and she is safe.
    
    For 2b, the root has to become majority colored.
    First we show that Alice can always choose a color for any vertex \(v\) at level 1, different from the color of root.
    If not, then it is possible in two cases:
    \begin{itemize}
        \item If uncolored \(v\) has two majority colored children.
        Again Alice does not create them as shown before.
        As soon as Bob creates one of them, Alice colors \(v\) and she has a color available other than color of root.
        
        \item If uncolored \(v\) has one majority colored child, say \majcol{1}, and \(m+1\) of its children are of the same color, say 2, at level 2.
        Alice would not have colored any vertex at level 2 if there was an uncolored vertex at level 1.
        For choosing to color a vertex at level 2 over a vertex at level 1, our strategy requires a majority colored vertex at level 3, which does not happen by Lemma~\ref{L:not_maj_n-1}.
        Thus those \(m+1\) children and \majcol{1} were colored by Bob.
        But as soon as Bob colored first one of them, Alice would have colored the parent \(v\) by our strategy and she had a color available other than the color of root.       
    \end{itemize}

    Thus, Alice always has a color different from the root available for vertices at level 1, regardless of whether the root is majority colored or not.
    Consequently, \configuration{2b} cannot occur at \(v\), because if it did, it contradicts the fact that Alice has a color available for \(v\).

    This completes the proof.
\end{proof}

\section{Some natural strategies on general trees which fail for Alice}
\label{S:Natural_strategy}

In this section, we show why some natural strategies fail to be a winning strategy for Alice in the majority coloring game on trees with 3 colors.

First in Section~\ref{S:Game_coloring_Strategy}, the well-known marking game strategy of Faigle--Kern--Kierstead--Trotter~\cite{FaigleKernEtAl1993} on trees basically requires Alice to color the vertex which is closest to Bob's move and lies on the smallest subtree containing all previously colored vertices.
Figure~\ref{F:marking_game} below illustrates why this strategy does not help reduce the bound for the majority coloring game on trees.

Next, in Section~\ref{S:Strategy_color_parent}, we show why the strategy used to prove the bound for max degree 4 trees in Theorem~\ref{T:FSTTCSa} does not extend to general trees.

Lastly, as discussed in Section~\ref{S:Preliminaries}, a winning strategy for Alice on a general tree must address both types of critical configurations.
A natural way for Alice to decide which vertex to pick is to assign a ``criticality'' index to each vertex based on how many moves it will take for a vertex to become critical.
We discuss two such carefully designed parametric strategy variants and illustrate their failure on a specifically engineered trees in Section~\ref{S:lambda_strategy}.

\subsection{The coloring game strategy}
\label{S:Game_coloring_Strategy}
Faigle--Kern--Kierstead--Trotter~\cite{FaigleKernEtAl1993} showed that \(\mathrm{col}_g(T)\leq 4\), i.e. Alice has a strategy for the marking game such that whenever she chooses to color an uncolored vertex, it has at most 3 marked neighbors.
Since the same strategy can also be applied to the coloring game to show that \(\chi_g(T) \leq 4\), for the sake of convenience we will refer to their strategy as the \emph{coloring game strategy}.

Now note that for \configuration{2} with \(v\) as a critical vertex, we need at least \(m+3\) colored neighbors, where \(m=\half{\deg(v)}\). 
If $m \ge 1$, then \(m+3 \ge 4\) which Alice can avoid using the coloring game strategy.
Also, if \(m=0\), then the degree of \(v\) can be \(2m+1=1\), i.e. \(v\) is a leaf and the critical configuration cannot occur at a leaf.

This shows that using the coloring game strategy, Alice can avoid \configuration{2} on any tree.
Thus, a natural direction is to study to what extent it helps Alice in the majority coloring game.
We give an example showing that the strategy does not work for the majority coloring game.

Consider the majority coloring game on a tree shown in Figure~\ref{F:marking_game}.
Alice plays using the coloring game strategy.
Odd numbers show Alice's moves and even show Bob's.
\begin{figure}[h]
\centering
    \begin{tikzpicture}
        [level distance=5mm,
        font=\scriptsize,
        every node/.style={draw,circle, inner sep=2pt},
        level 1/.style={sibling distance=30mm},
        level 2/.style={sibling distance=30mm},
        level 3/.style={sibling distance=30mm},
        level 4/.style={sibling distance=20mm},
        level 5/.style={sibling distance=20mm},
        level 6/.style={sibling distance=10mm},
        line width=0.8pt,
        smooth]

        \node [fill=red!30]{1}
        child{node [fill=cyan!30] {3}
            child{node [fill=red!30]{7}
                child{node [fill=cyan!30] {6}
                child{node [fill=cyan!30] {5}
                child{node {u}
                        child{node [fill=red!30]{2}}
                        child{node [fill=green!30]{8}}}
                child{node {v}
                        child{node [fill=cyan!30]{4}}
                        child{node [inner sep=3.5pt]{}}}
                    }
                child{node {...}}
                }
                child{node {...}}}
            child{node {...}}
        }
        child{node {...}};
    \end{tikzpicture}
    \caption{Failure of coloring game strategy when applied to majority coloring game}
    \label{F:marking_game}
\end{figure}
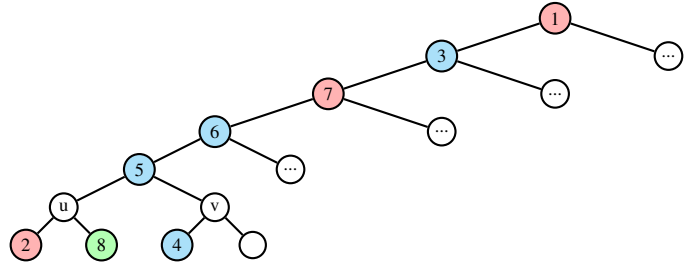

Before describing the game, we give a brief description of the strategy given in ~\cite{FaigleKernEtAl1993}.
Alice colors any vertex \(r\) and roots the tree on that vertex \(r\).
She maintains a sub-tree \(T_0\), which is initialized as \(T_0=\{r\}\).
Whenever Bob colors any vertex \(v\), denote by \(P\) the directed path from \(r\) to \(v\).
Alice after Bob's move colors the last vertex common between \(P\) and \(T_0\), if it is colored, she colors any uncolored vertex in \(T_0\) and if all vertices in \(T_0\) are colored, she colors any vertex \(u\) adjacent to \(T_0\).
She updates \(T_0\) as \(T_0 \coloneqq T_0 \cup P\) or \(T_0:=T_0 \cup P \cup \{u\}\) accordingly.

\emph{Description of game:} 
Alice colors 1 with red (arbitrary choice), and roots the tree at 1 and \(T_0=\{1\}\).
When Bob colors 2 with red, \(P=\{1,3,7,6,5,u,2\}\) and there's no uncolored vertex common in \(P\) and \(T_0\).
So Alice colors 3, which is a vertex adjacent to \(T_0\) and now \(T_0=\{1,3,7,6,5,u,2\}\).
Now Bob colors 4 with blue and this time \(P=\{1,3,7,6,5,v,4\}\).
This time, 5 is the last vertex common between \(P\) and \(T_0\) and is uncolored.
So Alice colors 5 with a color, say blue and updates \(T_0\) and now \(T_0=\{1,3,7,6,5,u,v,2,4\}\).
Bob then colors 6 with blue so that \(P=\{1,3,7,6\}\) and since 7 is the last vertex common between \(P\) and \(T_0\), Alice colors 7.
Now Bob colors 8 with green and wins the game as no legal color is left for \(u\).

Since 2 and 8 are leaves, their neighbors cannot have the same color, otherwise it violates majority condition.
So red and green are no longer available for \(u\).
Also coloring \(u\) with blue it violates majority condition for 5 because degree of 5 is three and thus only one neighbor of 5 can have same color as 5 as per the majority condition.
Hence Bob wins the game.

Even if Alice used red at 5, Bob could have played in exact same manner and then could have done green at the other child of \(v\), leaving no color for \(v\).
And if Alice chose green at 3, then also Bob could have played as above and forbidden all three colors at \(u\) or \(v\).

\subsection{A natural extension of the strategy of Theorem~\ref{T:FSTTCS}(\ref{T:FSTTCSa})}
\label{S:Strategy_color_parent}

One natural strategy that one can think of for Alice in the majority coloring game is to color the parent of any majority colored vertex created by Bob instantly and to use minimum used colors on siblings.
This was the core strategy used for proving Theorem~\ref{T:FSTTCS}(\ref{T:FSTTCSa}) in~\cite{ChawdaNanotiSankarnarayanan2026}).
However, such a naive strategy fails to extend as a winning strategy for Alice for arbitrary trees.

Consider the majority coloring game played by Alice using this strategy on a section of tree as shown in figure~\ref{F:color_p}.
Here, odd numbers depict Bob's move and even depict Alice's moves.
The vertices \(w,1,5,7,9,11\) are leaves.
Note that coloring a leaf makes it majority colored instantly as the degree of leaf is one.
The gadget shown in Figure~\ref{F:color_pb} is present at each of \(u_i's\) in Figure~\ref{F:color_pa}.

\emph{Game description:} Suppose that many copies of this gadget are present in the tree and Alice plays the first move somewhere far.
Bob starts by coloring the gadget \(u_1\).
Bob successively colors 2,4 and 6 with red which makes them majority colored and Alice thus colors their parents, alternating between the two available colors.
Clearly, 2 out of 3 children of \(u_1\) will have same color, here green.
Bob then uses the lesser used color, here blue on 8 which is a leaf.
Now Alice has to color \(u_1\) and only 2 colors are available:red and green.
Green makes \(u_1\) majority colored and red doesn't.
(If Alice kept 2 children of \(u_1\) blue and 1 green, then red and blue would be available with blue making it majority colored.)

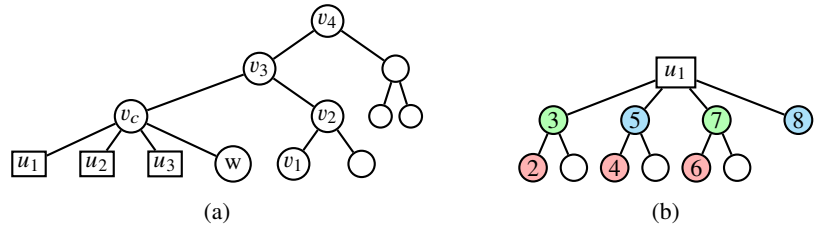
\begin{figure}[h]
    \centering
\begin{subfigure}[b]{0.48\textwidth}
    \centering
    \begin{tikzpicture}
        [level distance=7mm,
        every node/.style={draw,circle, inner sep=2pt},
        level 1/.style={sibling distance=20mm},
        level 2/.style={sibling distance=20mm},
        level 3/.style={sibling distance=10mm},
        line width=0.8pt,
        scale=0.9,
        smooth]
    
        \node [inner sep=1pt]{\(v_4\)}
        child { node [inner sep=1pt]{\(v_3\)}
            child { node [xshift=-8mm, inner sep=1pt] {\(v_c\)}
                child { node [shape=rectangle, fill=white, inner sep=2pt, draw] {$u_1$} }
                child { node [shape=rectangle, fill=white, inner sep=2pt, draw] {$u_2$} }
                child { node [shape=rectangle, fill=white, inner sep=2pt, draw] {$u_3$} }
                child { node {w} }
            }
            child { node [inner sep=1pt]{\(v_2\)}
                child { node [inner sep=1pt]{\(v_1\)} }
                child { node [inner sep=3.5pt]{} }
            }
        }
        child { node [inner sep=3.5pt]{}
            child { node [xshift=7mm, inner sep=3pt]  {} }
            child { node [xshift=-7mm, inner sep=3pt] {} }
        };
    \end{tikzpicture}
    \caption{}
    \label{F:color_pa}
\end{subfigure}
\begin{subfigure}[b]{0.48\textwidth}
\centering
    \begin{tikzpicture}
        [level distance=7mm,
        every node/.style={draw,circle, inner sep=2pt},
        level 1/.style={sibling distance=12mm},
        level 2/.style={sibling distance=6mm},
        line width=0.8pt,
        scale=0.9,
        smooth]
    
        \node [shape=rectangle, draw, inner sep=3pt, fill=white] {$u_1$}
        child { node [fill=green!30, inner sep=1pt] {3}
                        child { node [fill=red!30, inner sep=1pt] {2} }
                        child { node [inner sep=3.5pt]{} }
                    }
                    child { node [fill=cyan!30, inner sep=1pt] {5}
                        child { node [fill=red!30, inner sep=1pt] {4} }
                        child { node [inner sep=3.5pt]{} }
                    }
                    child { node [fill=green!30, inner sep=1pt] {7}
                        child { node [fill=red!30, inner sep=1pt] {6} }
                        child { node [inner sep=3.5pt]{} }
                    }
                    child { node [fill=cyan!30, inner sep=1pt] {8} };
    \end{tikzpicture}
    \caption{}
    \label{F:color_pb}
\end{subfigure}
\caption{Gadgets for counter-example to the extension of strategy of Theorem~\ref{T:FSTTCS}(\ref{T:FSTTCSa})}
\label{F:color_p}
\end{figure}

Suppose she chooses red at \(u_1\).
Then Bob starts coloring gadget at \(u_2\) in the same way as \(u_1\).
Again, she has to choose from two colors at \(u_2\), one of them makes it majority colored and the other is red.
Supposes she again chooses red and then again at \(u_3\) when Bob colors the gadget.
Alice thus ends up coloring 3 neighbors of \(v_c\) with red which forbids red at \(v_c\) because degree of \(v_c\) is 5 and at most 2 of its neighbors can have the same color as itself.
Bob then goes and colors the leaf \(v_1\) with red and Alice colors its parent with a legal color, say blue.
Bob colors \(v_3\) with blue which makes \(v_3\) with majority colored (as its degree is 3) and Alice colors \(v_4\).
Now red and blue both are forbidden at \(v_c\), so Bob colors the leaf \(w\) with green and hence Alice loses the game as no legal color is left for \(v_c\).
The game is shown in Figure~\ref{F:color_m+1}.

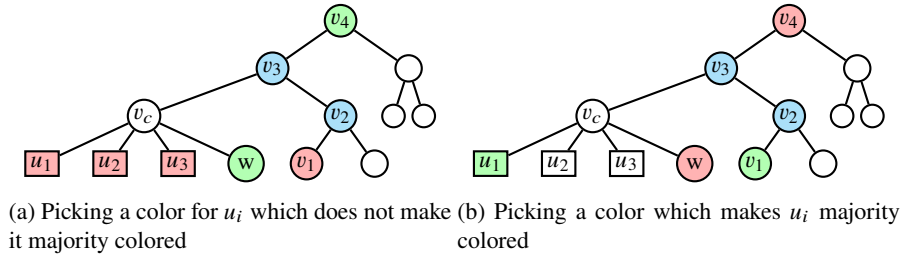
\begin{figure}[h]
    \centering
\begin{subfigure}[b]{0.48\textwidth}
    \centering
    \begin{tikzpicture}
        [level distance=7mm,
        every node/.style={draw,circle, inner sep=2pt},
        level 1/.style={sibling distance=20mm},
        level 2/.style={sibling distance=20mm},
        level 3/.style={sibling distance=10mm},
        line width=0.8pt,
        scale=0.9,
        smooth]
    
        \node [inner sep=1pt, fill=green!30]{\(v_4\)}
        child { node [inner sep=1pt, fill=cyan!30]{\(v_3\)}
            child { node [xshift=-8mm, inner sep=1pt] {\(v_c\)}
                child { node [shape=rectangle, fill=red!30, inner sep=2pt, draw] {$u_1$} }
                child { node [shape=rectangle, fill=red!30, inner sep=2pt, draw] {$u_2$} }
                child { node [shape=rectangle, fill=red!30, inner sep=2pt, draw] {$u_3$} }
                child { node [fill=green!30]{w} }
            }
            child { node [inner sep=1pt, fill=cyan!30]{\(v_2\)}
                child { node [inner sep=1pt, fill=red!30]{\(v_1\)} }
                child { node [inner sep=3.5pt]{} }
            }
        }
        child { node [inner sep=3.5pt]{}
            child { node [xshift=7mm, inner sep=3pt]  {} }
            child { node [xshift=-7mm, inner sep=3pt] {} }
        };
    \end{tikzpicture}
    \caption{Picking a color for \(u_i\) which does not make it majority colored}
    \label{F:color_m+1}
\end{subfigure}
\begin{subfigure}[b]{0.48\textwidth}
\centering
    \begin{tikzpicture}
        [level distance=7mm,
        every node/.style={draw,circle, inner sep=2pt},
        level 1/.style={sibling distance=20mm},
        level 2/.style={sibling distance=20mm},
        level 3/.style={sibling distance=10mm},
        line width=0.8pt,
        scale=0.9,
        smooth]
    
        \node [inner sep=1pt, fill=red!30]{\(v_4\)}
        child { node [inner sep=1pt, fill=cyan!30]{\(v_3\)}
            child { node [xshift=-8mm, inner sep=1pt] {\(v_c\)}
                child { node [shape=rectangle, fill=green!30, inner sep=2pt, draw] {$u_1$} }
                child { node [shape=rectangle, fill=white, inner sep=2pt, draw] {$u_2$} }
                child { node [shape=rectangle, fill=white, inner sep=2pt, draw] {$u_3$} }
                child { node [fill=red!30]{w} }
            }
            child { node [inner sep=1pt, fill=cyan!30]{\(v_2\)}
                child { node [inner sep=1pt, fill=green!30]{\(v_1\)} }
                child { node [inner sep=3.5pt]{} }
            }
        }
        child { node [inner sep=3.5pt]{}
            child { node [xshift=7mm, inner sep=3pt]  {} }
            child { node [xshift=-7mm, inner sep=3pt] {} }
        };
    \end{tikzpicture}
    \caption{Picking a color which makes \(u_i\) majority colored}
    \label{F:color_maj}
\end{subfigure}
\caption{Consequences of coloring \(u_i\)}
\label{F:color_p_conseq}
\end{figure}

Now suppose Alice colors \(u_1\) with green instead of red.
This makes \(u_1\) majority colored and green is forbidden at \(v_c\).
Bob then colors \(v_1\) with green which makes it majority colored and Alice colors its parent \(v_2\) with a legal color, say blue.
Bob colors \(v_3\) with the same color, which makes it majority colored and Alice colors its parent.
Again two colors, blue and green, are forbidden at \(v_c\) and forbids the third color, red, by coloring the leaf \(w\).
Alice thus loses the game, which is shown in Figure~\ref{F:color_maj}.

This shows that the natural strategy of always coloring the parent does not work.
Note that in both Figure~\ref{F:color_m+1} and~\ref{F:color_maj} Alice ends up forbidding a color at \(v_c\) which helps Bob.

\subsection{The parametric strategies}
\label{S:lambda_strategy}

Now we discuss a comprehensive parametric approach that involved defining a parameter for each vertex and Alice deciding her move based on value of parameter for each vertex.
However, it failed to yield a winning strategy for Alice.

Firstly, consider the majority coloring game with two colors on the tree shown in Figure~\ref{F:counter1}.
If Alice begins by coloring vertex 1 red, Bob wins by coloring vertex 2 green.
However, Alice has a winning strategy on this tree.
Thus, the choice of the vertex that Alice attacks is an important part of her strategy.
Similarly, for the majority coloring game with three colors, we ask whether the choice of vertex affects Alice's winning strategy.
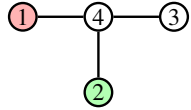
\begin{figure}[h!]   
    \centering
    \begin{tikzpicture}[
            every node/.style={circle, draw, fill=white, inner sep=1pt},
            line width=0.9pt,
            smooth
        ]
       
        \node (3) at (3,0) [fill=red!30] {1};
        \node (4) at (4,0) {4};
        \node (7) at (4,-1) [fill=green!30] {2};
        \node (5) at (5,0) {3};
        
        \draw (3)--(4)--(5);
        \draw (4)--(7);
    \end{tikzpicture}
    \caption{A counter example which shows the impact of Alice starting vertex. Note that the vertex $1$ is colored red by Alice and the vertex $2$ is colored green by Bob.}
    \label{F:counter1}
\end{figure}

To address this problem, we introduce a parameter \(\lambda(v)\) for each vertex \(v\) of the tree, denoting the number of moves required to force \(v\) into one of the configurations shown in Figure~\ref{crit_config}. More precisely, we define two parameters, \(\lambda_1\) and \(\lambda_2\), denoting the minimum number of moves required to force \(v\) into \configuration{1} in Figure~\ref{F:config1} and \configuration{2} in Figure~\ref{F:config2}, respectively. Alice uses \(\lambda(v)\) to decide which vertex to color next. The remainder of this section presents two methods for calculating these parameters, with the second improving on the first.

\textbf{Preliminaries}
Let \(G=(V,E)\) be a graph.
Throughout the algorithm, the neighborhood \(N(v)\) of each vertex \(v\) is taken from the \emph{original} graph \(G\), even as colors are assigned progressively.
Let \(\deg(u)\) denote the degree of a vertex \(u\) in \(G\), and let \(\deg^u(u)\) denote the number of its currently uncolored neighbors, recomputed after each coloring step.
The palette is $\colors=\{r,g,b\}$. For each vertex $v\in V$, let \(\mathrm{color}(v)\) denote its color, where \( \mathrm{color}(v) \in \colors\) if \(v\) is colored and \(\mathrm{color}(v)=\phi\) otherwise.
Also, recall that a colored vertex \(u\) is said to be \emph{majority colored} if the
number of neighbors of \(u\) in the graph \(G\) having the
same color as \(u\) is exactly \(\half{\deg(u)}\).

\subsubsection{The \(\lambda\)-strategy:} In this section, we introduce the \(\lambda\)-strategy, which involves calculating the parameter \(\lambda(v)\) for each vertex \(v\) of the tree. Before defining these parameters, we introduce the terminology and notation used throughout this section.
\begin{definition}[Eligible colors]
\label{def:eligible-colors}
For any vertex \(v\), whether colored or uncolored, a color \(\in\{r,g,b\}\) is \emph{eligible} if no neighbor of \(v\) is majority colored with color \(c\). Thus,
\[
E(v)=\{r,g,b\}\setminus
\{
\operatorname{color}(u):
u\in N(v),\
u\text{ is majority colored}
\}.
\]
Define
\[
c(v)=|E(v)|.
\]
\end{definition}

As discussed at the beginning of this section, the parameters $\lambda_1(v)$ and $\lambda_2(v)$ denote the minimum number of moves required, regardless of whether Alice or Bob makes them, to place \(v\) in the configurations shown in Figures~\ref{F:config1} and~\ref{F:config2}, respectively.
In the first configuration, \(v\) is adjacent to three majority colored vertices, each with a different color.
The number of moves required to make a vertex \(u\) majority colored is $\lfloor \deg^u(u)/2\rfloor+1$: coloring \(u\) and $\lfloor \deg^u(u)/2\rfloor$ of its neighbors with the same color makes \(u\) majority colored. We therefore calculate $\lambda_1(v)$ for any vertex \(v\) as follows.

For every \emph{uncolored} vertex \(v\), define
\[
\lambda_1(v)=
\begin{cases}
\displaystyle
\min_{\substack{S\subseteq N(v)\\|S|=c(v)}}
\left(
\sum_{u\in S}\left\lfloor\frac{\deg^u(u)}{2}\right\rfloor+1
\right),
&
|N(v)|\ge c(v),\\[2.5ex]
\infty,
&
|N(v)|<c(v).
\end{cases}
\]

Similarly, in the second configuration, \(v\) is adjacent to two majority colored vertices with different colors and has $\lfloor \deg^u(v)/2\rfloor+1$ neighbors colored with the remaining color.
We therefore calculate $\lambda_2(v)$ for any vertex \(v\) as follows.

\[
\lambda_2(v)=
\begin{cases}
\displaystyle
\min_{\substack{S\subseteq N(v)\\|S|=c(v)-1}}
\left(
\sum_{u\in S}\left\lfloor\frac{\deg^u(u)}{2}\right\rfloor+1
\right)
+\left\lfloor\frac{\deg^u(v)}{2}\right\rfloor+1,
&
\begin{array}{l}
c(v)\ge 1,\\
|N(v)|\ge c(v)-1,
\end{array}
\\[4ex]
\infty,
&
\text{otherwise.}
\end{cases}
\]

Finally, for any vertex \(v\), we define $\lambda(v)=\min\{\lambda_1(v),\lambda_2(v)\}$, which denotes the minimum number of moves required to place \(v\) in either of the two configurations.

The parameters $\lambda_1(v)$, $\lambda_2(v)$, and \(\lambda(v)\) are updated after each round for every vertex \(v\).
Therefore, uncolored neighbors are included when calculating these parameters.

\textbf{Round behavior:} Alice colors the vertices in the following manner:
\begin{enumerate}
    \item Select an uncolored vertex with the smallest \(\lambda(v)\).
    If several vertices attain the minimum, choose
    one uniformly at random.

    \item Let \(v\) be the selected vertex. Choose a color from
    \(E(v)\) that satisfies both of the following conditions:
    \begin{enumerate}
        \item It occurs least frequently among already colored
        vertices at graph distance at most two from \(v\).

        \item It appears on at most half of the already colored
        neighbors of \(v\).
    \end{enumerate}

    If several colors satisfy both conditions, choose one uniformly at random.

    If there are no colored vertices within distance two of \(v\), every eligible color has count zero.
    Choose one of these
    colors uniformly at random.

    If \(E(v)=\varnothing\) for any uncolored vertex \(v\), the algorithm reports \textbf{Bob Wins} and terminates.
\end{enumerate}

Unfortunately, the strategy described above does not guarantee a win for Alice.
We present a counterexample in Figure~\ref{F:game11august}.
\begin{figure}[h]
\centering
\begin{tikzpicture}[
    vertex/.style={circle, draw, minimum size=3.5mm, line width=0.8pt, inner sep=1pt, font=\fontsize{6}{6}\selectfont\bfseries},
    line width=0.9pt,
    smooth,
    scale=0.8
]
\begin{scope}
\def\vertexdata{
    0/0.0/0.0,
    1/-0.8/1.2/g,
    2/1.3/-0.2/b,
    3/0.8/-1.6,
    4/-0.2/2.4,
    5/-1.9/0.8/g,
    6/-1.8/2.0/g,
    7/-1.2/0.7,
    8/2.1/-0.4,
    9/2.5/0.1,
    10/1.8/0.7/r,
    11/-1.2/-2.1,
    12/1.0/-2.9,
    13/2.5/-1.6,
    14/0.2/-1.3,
    15/2.1/3.2/b,
    16/0.8/3.7/b,
    17/-0.9/3.6/b,
    18/1.3/2.4/r,
    19/-3.2/2.5/r,
    20/-3.6/1.0,
    21/-3.5/-0.5,
    22/-2.8/0.6,
    23/-1.0/2.2,
    24/-1.8/2.9,
    25/-2.4/2.1,
    26/-3.0/-2.0,
    27/-2.6/-2.8,
    28/-1.6/-3.2,
    29/-1.7/-1.7,
    30/0.3/-4.2,
    31/1.5/-4.2,
    32/2.5/-3.2,
    33/1.9/-2.8,
    34/3.2/-2.5,
    35/4.0/-1.5,
    36/3.9/-0.7,
    37/3.0/-0.5,
    38/3.1/3.1/r,
    42/1.9/4.2,
    43/1.1/4.7,
    44/0.1/4.7,
    45/-0.8/4.7,
    46/-1.5/4.5,
    47/-2.2/4.1,
    48/-2.6/3.7,
    49/-3.5/3.5,
    50/-4.1/3.0,
    51/-4.5/2.5,
    52/-4.8/1.8,
    53/-4.8/1.1,
    54/-4.8/0.3,
    55/-4.8/-0.3,
    56/-4.4/-1.0,
    66/-3.9/-1.8,
    67/-4.1/-2.4,
    68/-3.6/-2.8,
    69/-3.8/-3.3,
    70/-3.3/-3.8,
    71/-2.8/-4.1,
    72/-2.2/-4.1,
    73/-1.6/-4.5,
    74/-0.9/-4.5,
    75/-0.5/-4.8,
    76/0.0/-5.1,
    77/0.8/-5.0,
    78/1.2/-5.1,
    79/1.8/-5.0,
    80/2.7/-4.6,
    81/3.1/-4.2,
    82/3.5/-3.7,
    83/3.9/-3.2,
    84/4.1/-2.8,
    85/4.7/-2.4,
    86/4.7/-1.9,
    87/5.1/-1.4,
    88/5.3/-0.8,
    89/5.2/-0.3,
    90/5.0/0.2,
    91/4.7/0.7,
    92/4.2/1.0,
    93/3.1/4.3/g,
    94/4.8/3.8,
    95/3.9/4.8,
    96/3.0/5.2
    }
\foreach \id/\xx/\yy/\state in \vertexdata {
    \coordinate (p\id) at (\xx,\yy){};
}
\foreach \id/\xx/\yy/\state in \vertexdata {
    \def\vertexfill{gray!30}
    \ifdefstring{\state}{r}{\def\vertexfill{red!40}}{}
    \ifdefstring{\state}{g}{\def\vertexfill{green!40}}{}
    \ifdefstring{\state}{b}{\def\vertexfill{blue!40}}{}

    \node[vertex,fill=\vertexfill] (v\id)
        at (p\id) {\id};
}

\foreach \parent/\children in {
    0/{1,2,3},
    1/{4,5,6,7},
    2/{8,9,10},
    3/{11,12,13,14},
    4/{15,16,17,18},
    5/{19,20,21,22},
    6/{23,24,25},
    11/{26,27,28,29},
    12/{30,31,32,33},
    13/{34,35,36,37},
    15/{38,93},
    16/{42,43,44},
    17/{45,46,47},
    19/{48,49,50},
    20/{51,52,53},
    21/{54,55,56},
    26/{66,67},
    27/{68,69,70},
    28/{71,72,73,74},
    30/{75,76,77},
    31/{78,79,80},
    32/{81,82,83},
    34/{84,85},
    35/{86,87,88},
    36/{89,90,91,92},
    93/{94,95,96},
    }{
        \foreach \child in \children {
            \draw (v\parent)--(v\child); 
            }
    }
\node[
    vertex,
    fill=gray!30,
    draw=orange,
    line width=1.8pt,
    minimum size=5.5mm
] at (p4) {4};
\end{scope}
\end{tikzpicture}
\caption{A Counterexample for \(\lambda\)-Strategy}
\label{F:game11august}
\end{figure}
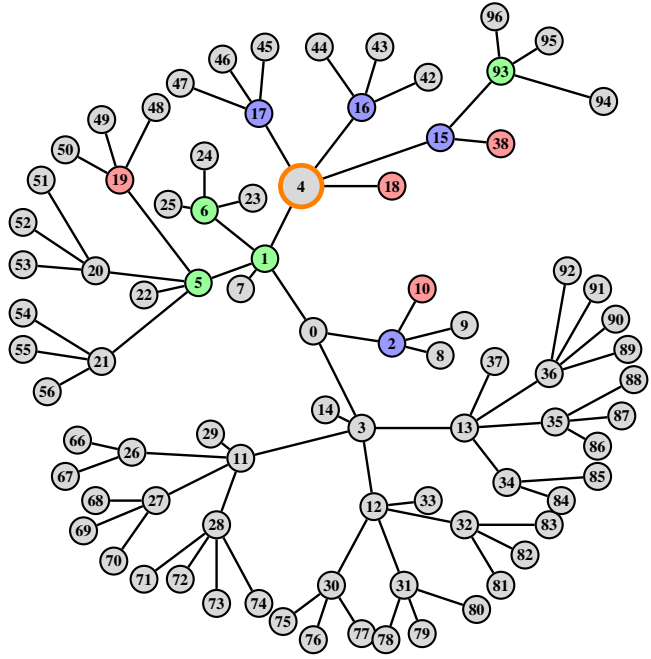

The issue with the previous strategy is that the parameter \(\lambda(v)\) for each vertex \(v\) is calculated solely based on the uncolored degrees of its neighbors, without considering the colors already used in their neighborhoods.
Consequently, a vertex that may create a problem for Alice can be overlooked. 

\begin{table}[ht]
\centering
\begin{tabular}{|c|c|c|c|c|}
\hline
Move  & Player & Vertex & Color \\ \hline
1   & Alice & 93 & Green (g) \\ \hline
2  & Bob   & 38 & Red (r)   \\ \hline
 3  & Alice & 16 & Blue (b)  \\ \hline
 4  & Bob   & 17 & Blue (b)  \\ \hline
 5  & Alice & 15 & Blue (b)  \\ \hline
 6  & Bob   & 18 & Red (r)   \\ \hline
 7  & Alice & 6  & Green (g) \\ \hline
 8  & Bob   & 10 & Red (r)   \\ \hline
9  & Alice & 2  & Blue (b)  \\ \hline
 10 & Bob   & 1  & Green (g) \\ \hline
 11 & Alice & 19 & Red (r)   \\ \hline
 12 & Bob   & 5  & Green (g) \\ \hline
\end{tabular}
\vspace{0.3cm}
\caption{Vertex coloring moves by Alice and Bob for the tree in Figure~\ref{F:game11august}.}
\label{tab:vertex-coloring}
\end{table}

The tree in Figure~\ref{F:game11august} provides an example in which Alice can lose by following the \(\lambda\)-strategy.
Table~\ref{tab:vertex-coloring} lists the vertices and colors chosen by Alice and Bob at each move.
After move 6, when Bob colors vertex 18 red, green is the only eligible color for vertex 4.
However, the \(\lambda\)-strategy does not force Alice to color vertex 4 green on her next move.
This is because, after move 6, $\lambda(4)=3$, while other vertices, such as vertices 6 and 19, also have a \(\lambda\)-value of 3.
The strategy allows ties for the minimum \(\lambda\)-value to be broken randomly.
Thus, Alice chooses vertex $6$ on her next move, which eventually leads to her loss.

The issue is that, although we calculate the number of moves required for a vertex to reach configuration 1 or configuration 2, this calculation does not account for the colors of its neighbors.
To illustrate this, consider the same example. Vertex $4$ is more likely to reach configuration 2 after Bob colors vertex 18 red.
However, because vertex $4$ has more neighbors, the current strategy does not capture this threat.
To incorporate information about the colors of neighboring vertices, we propose a different strategy in the next section.

\subsubsection{The \((\lambda, \mu)\)-strategy:}

To address the issue discussed in the previous section, we introduce the parameter \(\mu_i(v)\) for each vertex \(v\) and color $i \in \colors$.
This parameter denotes the minimum number of moves required to make \(v\) majority colored with color \(i\).
The \(\lambda_1\) and \(\lambda_2\) values for each vertex are then calculated using the \(\mu\)-values of its neighbors.

For each vertex \(v\) and color $i \in \colors$, we initialize \(\mu_i(v)\) as follows:
\[
\mu_i(v)=
\begin{cases}
1,
& \deg(v)=1,\\[1ex]
\infty,
& \deg(v)\neq 1
  \text{ and }\ell(v)>
  \left\lfloor\frac{\deg(v)}{2}\right\rfloor,\\[1ex]
\displaystyle
\left\lfloor\frac{\deg(v)}{2}\right\rfloor+1,
& \text{otherwise},
\end{cases}
\]
where $\ell(v)=|\{u\in N(v):\deg(u)=1\}|$.

For a vertex \(v\) to be majority colored, the number of its non-pendant neighbors must be at least half its degree.
This is why we initialize $\mu_i(v)=\infty$ in the second case.

\begin{definition}[Eligible colors]
For any vertex \(v\), colored or uncolored, a color
\(i\in\colors\) is \emph{eligible} if both of the following
conditions hold:
\begin{enumerate}
    \item No neighbor of \(v\) is majority colored with color \(i\).
    \item \(\mu_i(v)\neq 0\).
\end{enumerate}
Thus,
\[
E(v)=
\left\{
i\in\colors:
\begin{array}{l}
\mu_i(v)\neq 0,\text{ and no neighbor of }v\\
\text{is majority colored with color }i
\end{array}
\right\}.
\]
 Define
\[
c(v)=|E(v)|.
\]
\end{definition}

For a vertex \(v\), \(\mu_i(v)\) denotes the minimum number of moves required to make \(v\) majority colored with color \(i\).
Taking the minimum over all eligible colors of \(v\) gives the minimum number of moves required to make \(v\) majority colored. Formally, define
\[
\mu(v)=\min_{i \in E(v)} \mu_i(v).
\]

Using the parameter \(\mu\), we define the \(\lambda_1\) and \(\lambda_2\)-values.
The \(\lambda_1\)-value of a vertex \(v\) denotes the minimum number of moves required for \(v\) to have three majority colored neighbors of distinct colors.
Similarly, its \(\lambda_2\)-value denotes the minimum number of moves required for \(v\) to have two majority colored neighbors of distinct colors and at least half of its neighbors colored with the remaining color.

Formally, for every non-pendant vertex \(v\), define
\[
\lambda_1(v)=
\begin{cases}
\displaystyle
\min_{\substack{A\subseteq N(v)\\|A|=c(v)}}
\sum_{u\in A}\mu(u),
& |N(v)|\ge c(v),\\[2.5ex]
\infty,
& |N(v)|<c(v).
\end{cases}
\]
Similarly, define
\[
\lambda_2(v)=
\begin{cases}
\displaystyle
\min_{\substack{A\subseteq N(v)\\|A|=c(v)-1}}
\left(\sum_{u\in A}\mu(u)+\mu(v)\right),
&
\begin{array}{l}
c(v)\ge 1,\\
|N(v)|\ge c(v)-1,
\end{array}\\[3ex]
\infty,
& \text{otherwise}.
\end{cases}
\]

Finally, for any vertex \(v\), we define $\lambda(v)=\min\{\lambda_1(v),\lambda_2(v)\}$, which denotes the minimum number of moves required to place \(v\) in either of the two configurations.

We use the following conventions:
\begin{itemize}
    \item A sum over the empty subset is \(0\). Consequently, when \(c(v)=0\),
    \[
    \lambda_1(v)=0,
    \qquad
    \lambda_2(v)=\infty.
    \]
    \item For every pendant vertex \(v\), set
    \[
    \lambda(v)=\infty.
    \]
\end{itemize}
Whenever either player assigns color \(i\) to an uncolored vertex \(u\), the \(\mu\)-values are updated as follows:
\begin{enumerate}
    \item If \(u\) is non-pendant, set
    \[
    \mu_i(v)\leftarrow\mu_i(v)-1
    \qquad\text{for every }v\in N(u).
    \]
    Also set
    \[
    \mu_i(u)\leftarrow\mu_i(u)-1,
    \qquad
    \mu_j(u)\leftarrow\infty
    \quad\text{for every }j\in\colors\setminus\{i\}.
    \]
    \item If \(u\) is pendant, let \(v\) be its unique neighbor and set
    \[
    \mu_i(v)\leftarrow 0.
    \]
\end{enumerate}
The \(\mu\) and \(\lambda\) values are also updated, as they depend on the $\mu_i$-values.

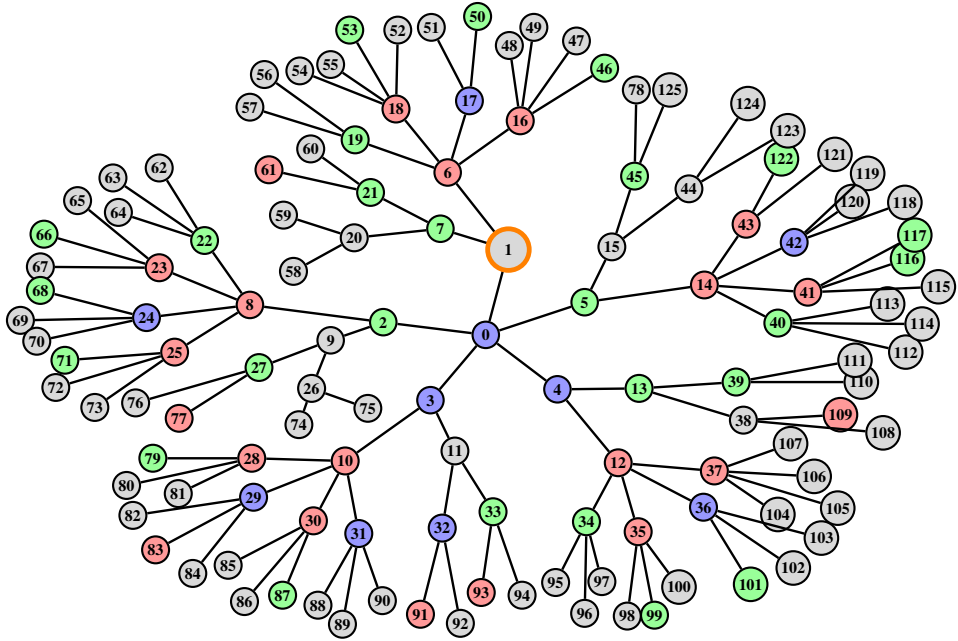
\begin{figure}[h]
\centering
\begin{tikzpicture}[
    x=0.015cm,
    y=-0.015cm,
    vertex/.style={circle, draw, minimum size=3.5mm, line width=0.8pt, inner sep=1pt, font=\fontsize{6}{6}\selectfont\bfseries},
    line width=0.9pt,
    smooth,
    scale=0.8
]
\path[use as bounding box] (180,130) rectangle (1235,830);
\begin{scope}
\def\vertexdata{
    0/714/495/b,
    1/739/401/,
    2/601/481/g,
    3/653/567/b,
    4/793/555/b,
    5/823/459/g,
    6/672/316/r,
    7/664/378/g,
    8/454/462/r,
    9/543/501/,
    10/559/634/r,
    11/680/623/,
    12/860/636/r,
    13/883/554/g,
    14/955/440/r,
    15/853/398/,
    16/752/259/r,
    17/696/237/b,
    18/615/246/r,
    19/570/280/g,
    20/569/388/,
    21/587/338/g,
    22/404/391/g,
    23/354/421/r,
    24/340/476/b,
    25/371/514/r,
    26/522/554/,
    27/464/531/g,
    28/456/631/r,
    29/458/674/b,
    30/524/700/r,
    31/574/714/b,
    32/666/708/b,
    33/722/690/g,
    34/825/701/g,
    35/882/712/r,
    36/954/685/b,
    37/966/645/r,
    38/998/589/,
    39/990/547/g,
    40/1036/482/g,
    41/1069/448/r,
    42/1054/393/b,
    43/1001/373/r,
    44/938/334/,
    45/878/320/g,
    46/845/201/g,
    47/814/171/,
    48/740/176/,
    49/767/156/,
    50/705/144/g,
    51/654/156/,
    52/617/159/,
    53/564/159/g,
    54/509/204/,
    55/543/197/,
    56/470/208/,
    57/454/244/,
    58/502/425/,
    59/491/363/,
    60/521/290/,
    61/475/313/r,
    62/354/311/,
    63/303/322/,
    64/309/360/,
    65/264/347/,
    66/227/384/g,
    67/223/420/,
    68/223/446/g,
    69/201/479/,
    70/219/502/,
    71/250/523/g,
    72/240/553/,
    73/283/575/,
    74/508/594/,
    75/584/576/,
    76/328/570/,
    77/376/587/r,
    78/880/225/,
    79/347/631/g,
    80/318/661/,
    81/375/670/,
    82/325/694/,
    83/350/732/r,
    84/391/758/,
    85/430/749/,
    86/448/789/,
    87/490/782/g,
    88/529/793/,
    89/556/814/,
    90/600/788/,
    91/642/803/r,
    92/686/813/,
    93/709/779/r,
    94/754/782/,
    95/791/768/,
    96/823/802/,
    97/842/767/,
    98/870/805/,
    99/900/804/g,
    100/927/773/,
    101/1006/770/g,
    102/1053/752/,
    103/1084/719/,
    104/1036/693/,
    105/1102/686/,
    106/1076/651/,
    107/1050/615/,
    108/1152/603/,
    109/1105/583/r,
    110/1128/547/,
    111/1121/523/,
    112/1177/514/,
    113/1157/461/,
    114/1195/483/,
    115/1213/442/,
    116/1179/411/g,
    117/1187/385/g,
    118/1176/349/,
    119/1135/317/,
    120/1118/348/,
    121/1099/296/,
    122/1040/301/g,
    123/1047/270/,
    124/1003/240/,
    125/917/225/r
}
\foreach \id/\xx/\yy/\state in \vertexdata {
    \coordinate (p\id) at (\xx,\yy){};
}
\foreach \id/\xx/\yy/\state in \vertexdata {
    \def\vertexfill{gray!30}
    \ifdefstring{\state}{r}{\def\vertexfill{red!40}}{}
    \ifdefstring{\state}{g}{\def\vertexfill{green!40}}{}
    \ifdefstring{\state}{b}{\def\vertexfill{blue!40}}{}

    \node[vertex,fill=\vertexfill] (v\id)
        at (p\id) {\id};
}

\foreach \parent/\children in {
    0/{1,2,3,4,5},
    1/{6,7},
    2/{8,9},
    3/{10,11},
    4/{12,13},
    5/{14,15},
    6/{16,17,18,19},
    7/{20,21},
    8/{22,23,24,25},
    9/{26,27},
    10/{28,29,30,31},
    11/{32,33},
    12/{34,35,36,37},
    13/{38,39},
    14/{40,41,42,43},
    15/{44,45},
    16/{46,47,48,49},
    17/{50,51},
    18/{52,53,54,55},
    19/{56,57},
    20/{58,59},
    21/{60,61},
    22/{62,63,64},
    23/{65,66,67},
    24/{68,69,70},
    25/{71,72,73},
    26/{74,75},
    27/{76,77},
    28/{79,80,81},
    29/{82,83,84},
    30/{85,86,87},
    31/{88,89,90},
    32/{91,92},
    33/{93,94},
    34/{95,96,97},
    35/{98,99,100},
    36/{101,102,103},
    37/{104,105,106,107},
    38/{108,109},
    39/{110,111},
    40/{112,113,114},
    41/{115,116,117},
    42/{118,119,120},
    43/{121,122},
    44/{123,124},
    45/{78,125},
    }{
        \foreach \child in \children {
            \draw (v\parent)--(v\child); 
            }
    }

\node[
    vertex,
    fill=gray!30,
    draw=orange,
    line width=1.8pt,
    minimum size=5.5mm
] at (p1) {1};
\end{scope}
\end{tikzpicture}
\caption{A Counterexample for $(\lambda, \mu)$-Strategy}
\label{F:game_20_august}
\end{figure}

\textbf{Round behavior:} In each round Alice chooses her moves in the following way:

\begin{enumerate}
    \item Select an uncolored vertex \(v\) with the smallest \(\lambda(v)\). If several vertices attain the minimum, choose one uniformly at random.

    \item Choose a color \(i\in E(v)\) that satisfies both of the following conditions:
    \begin{enumerate}
        \item \(\mu(u)\neq\mu_i(u)\) for every neighbor \(u\) of \(v\).
        \item If there is a neighbor \(u\) of \(v\) such that \(\mu(u)=\mu_i(u)\) for all \(i\in E(v)\), then \(i\) occurs on at most half of the already colored neighbors of \(v\) and does not make \(v\) majority colored after this move.
    \end{enumerate}

    If more than one color satisfies one of the above conditions, choose one that occurs least frequently among the already colored vertices at graph distance at most two from \(v\).
    If there are no colored vertices within distance two of \(v\), every eligible color has count zero. Choose one of these colors uniformly at random.

    If no color satisfies both conditions but \(E(v)\neq\varnothing\), choose an eligible color for \(v\) that occurs least frequently among the already colored vertices at graph distance at most two from \(v\).

    If \(E(v)=\varnothing\) for any uncolored vertex \(v\), the algorithm reports \textbf{Bob Wins} and terminates.
\end{enumerate}

Unfortunately, this strategy does not guarantee a win for Alice.
Figure~\ref{F:game_20_august} presents a counterexample, and Table~\ref{tab:vertex-coloring} lists the corresponding moves made by Alice and Bob.
Interestingly, no random tie-breaking among vertices with minimum \(\lambda\) was needed in this example.
The sequence of moves suggests that Bob can repeat the same pattern of play across different gadgets of the tree to create a problematic vertex. Moreover, many of Alice's moves indirectly contribute to Bob's victory.

\begin{table}[ht]
\centering
\begin{minipage}[t]{0.48\textwidth}
\centering
\begin{tabular}{|c|c|c|c|}
\hline
Move & Player & Vertex & Color \\ \hline
1  & Alice & 37  & Red (r)   \\ \hline
2  & Bob   & 101 & Green (g) \\ \hline
3  & Alice & 36  & Blue (b)  \\ \hline
4  & Bob   & 99  & Green (g) \\ \hline
5  & Alice & 35  & Red (r)   \\ \hline
6  & Bob   & 12  & Red (r)   \\ \hline
7  & Alice & 34  & Green (g) \\ \hline
8  & Bob   & 87  & Green (g) \\ \hline
9  & Alice & 30  & Red (r)   \\ \hline
10 & Bob   & 83  & Red (r)   \\ \hline
11 & Alice & 29  & Blue (b)  \\ \hline
12 & Bob   & 79  & Green (g) \\ \hline
13 & Alice & 28  & Red (r)   \\ \hline
14 & Bob   & 10  & Red (r)   \\ \hline
15 & Alice & 31  & Blue (b)  \\ \hline
16 & Bob   & 71  & Green (g) \\ \hline
17 & Alice & 25  & Red (r)   \\ \hline
18 & Bob   & 68  & Green (g) \\ \hline
19 & Alice & 24  & Blue (b)  \\ \hline
20 & Bob   & 66  & Green (g) \\ \hline
21 & Alice & 23  & Red (r)   \\ \hline
22 & Bob   & 8   & Red (r)   \\ \hline
23 & Alice & 22  & Green (g) \\ \hline
24 & Bob   & 46  & Green (g) \\ \hline
25 & Alice & 16  & Red (r)   \\ \hline
26 & Bob   & 50  & Green (g) \\ \hline
27 & Alice & 17  & Blue (b)  \\ \hline
28 & Bob   & 53  & Green (g) \\ \hline
29 & Alice & 18  & Red (r)   \\ \hline
\end{tabular}
\end{minipage}\hfill
\begin{minipage}[t]{0.48\textwidth}
\centering
\begin{tabular}{|c|c|c|c|}
\hline
Move & Player & Vertex & Color \\ \hline
30 & Bob   & 6   & Red (r)   \\ \hline
31 & Alice & 19  & Green (g) \\ \hline
32 & Bob   & 122 & Green (g) \\ \hline
33 & Alice & 43  & Red (r)   \\ \hline
34 & Bob   & 117 & Green (g) \\ \hline
35 & Alice & 42  & Blue (b)  \\ \hline
36 & Bob   & 116 & Green (g) \\ \hline
37 & Alice & 41  & Red (r)   \\ \hline
38 & Bob   & 14  & Red (r)   \\ \hline
39 & Alice & 40  & Green (g) \\ \hline
40 & Bob   & 77  & Red (r)   \\ \hline
41 & Alice & 27  & Green (g) \\ \hline
42 & Bob   & 61  & Red (r)   \\ \hline
43 & Alice & 21  & Green (g) \\ \hline
44 & Bob   & 125 & Red (r)   \\ \hline
45 & Alice & 45  & Green (g) \\ \hline
46 & Bob   & 93  & Red (r)   \\ \hline
47 & Alice & 33  & Green (g) \\ \hline
48 & Bob   & 91  & Red (r)   \\ \hline
49 & Alice & 32  & Blue (b)  \\ \hline
50 & Bob   & 109 & Red (r)   \\ \hline
51 & Alice & 39  & Green (g) \\ \hline
52 & Bob   & 13  & Green (g) \\ \hline
53 & Alice & 4   & Blue (b)  \\ \hline
54 & Bob   & 0   & Blue (b)  \\ \hline
55 & Alice & 5   & Green (g) \\ \hline
56 & Bob   & 3   & Blue (b)  \\ \hline
57 & Alice & 2   & Green (g) \\ \hline
58 & Bob   & 7   & Green (g) \\ \hline
\end{tabular}
\end{minipage}
\vspace{0.3cm}
\caption{Vertex coloring moves by Alice and Bob for the tree in Figure~\ref{F:game_20_august}.}
\label{tab:vertex-coloring_2}
\end{table}

\subsubsection{What do these parametric strategies tell us about the majority coloring game?}
In the first version, Alice chose to color the vertex that requires her attention most urgently, as measured by the number of moves it takes for it to be made critical.
In the second version, Alice chose to also measure the number of moves it takes for a vertex to become majority colored.
However, in either case, we were able to find gadgets to build a tree such that Alice's strategy fails to see that Bob is creating a problematic vertex in the future.

Perhaps the issue with these strategies is that Alice only focuses on the criticality of a vertex, and is allowing that to inform her decisions: one could consider defining a third parameter which measures the ``impact'' of coloring a vertex on its neighbors, and a strategy that also prioritizes minimizing the damage done by each move.
Despite the increase in complexity, we suspect that similar gadgets can be concocted which defeat such a strategy as well.
In general, we suspect that any strategy of Alice which ``sees'' only within a bounded radius of each vertex to determine her next move will be defeated by Bob on some sufficiently large tree.
In the concluding remarks, we propose a concrete (and simpler) version of this problem to better understand this behavior of the majority coloring game.

\subsection{Majority game chromatic number of infinite acyclic graphs}
\label{S:Infinite_Acyclic}

Recall that the \(d\)-relaxed coloring game is similar to the majority coloring game, the only difference being that the defect \(d\) at a vertex is defined globally whereas in majority coloring it depends on the degree of each vertex.

Now consider an infinite acyclic graph \(G\) with \(4 \le \delta(G) \le \Delta(G) \le 5\).
Note that \(\half{4}=\half{5}=2\) and hence the majority condition for such a graph translates to at most 2 neighbors of any given vertex having the same color as that vertex.
Consequently the majority coloring game on such a graph \(G\) is equivalent to the \(2\)-relaxed coloring game on \(G\).
Hence for such a graph, \(\mu_g(G)=\chi_g^2(G)\).
He--Wu--Zhu (cf.~\cite[Section 2]{HeWuZhu2004} proved that if \(F\) is a (finite) forest, and \(d\ge2\) then \(\chi^d_g(F)\le 2\).
Their strategy works by cutting the forest into trunks whose endpoints are the colored vertices and making an appropriate choice of trunk vertex for Alice to color.
It is not difficult to see that their strategy extends to infinite locally finite acyclic graphs, since at each partial stage of the game there are vertices only up to a finite depth which have been colored.
Thus for the graph \(G\) defined above, \(\mu_g(G)\le 2\); i.e., Alice has a winning strategy with two colors: just use the relaxed coloring game strategy.

This illustrates the key difference between the relaxed coloring game and the majority coloring game: when the defect is global, leaves play no significant role.
So \(\chi_g^d(F) \leq 2\) for \(d \geq 2\) for any locally finite forest, whether finite or infinite.
However, in the majority coloring game, when a leaf is colored it is automatically majority colored, so its neighbor cannot receive the same color as itself.
Thus, any winning strategy for Bob on a finite tree must crucially exploit the presence of leaves.
Concretely, we show below that there exist forests \(F\) in which all internal vertices have degree \(4\) or \(5\), but \(\mu_g(F) = 3\).
Such an example which is \(4\)-internal-regular is shown in Figure~\ref{F:degree_4}.

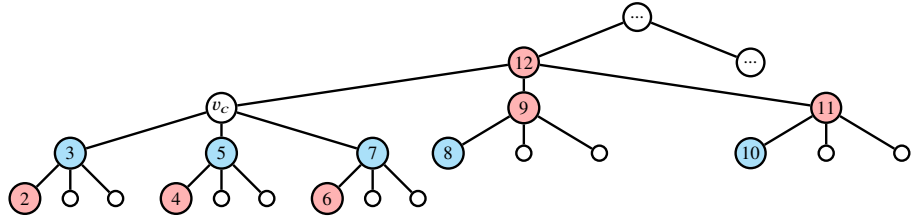
\begin{figure}[h]
\centering
    \begin{tikzpicture}
        [level distance=6mm,
        font=\scriptsize,
        every node/.style={draw,circle, inner sep=2pt},
        level 1/.style={sibling distance=30mm},
        level 2/.style={sibling distance=40mm},
        level 3/.style={sibling distance=20mm},
        level 4/.style={sibling distance=6mm},
        line width=0.9pt,
        smooth]

        \node {...}
        child{node [fill=red!30, inner sep=1pt] {12}
            child{node [inner sep=1pt]{\(v_c\)}
                child{node [fill=cyan!30] {3}
                    child{node [fill=red!30] {2}}
                    child{node {}}
                    child{node {}}
                    }
                child{node [fill=cyan!30]{5}
                    child{node [fill=red!30]{4}}
                    child{node {}}
                    child{node {}}}
                child{node [fill=cyan!30]{7}
                    child{node [fill=red!30]{6}}
                    child{node {}}
                    child{node {}}}
                }
            child{node [fill=red!30]{9}
                child{node [xshift=10mm, fill=cyan!30]{8}}
                child{node {}}
                child{node [xshift=-10mm]{}}}
            child{node [fill=red!30, inner sep=1pt]{11}
                child{node [xshift=10mm, fill=cyan!30, inner sep=1pt]{10}}
                child{node {}}
                child{node [xshift=-10mm]{}}}
        }
        child{node {...}};
    \end{tikzpicture}
    \caption{Alice loses using 2 colors on a 4 internal regular tree}
    \label{F:degree_4}
\end{figure}

\emph{Description of the game in Figure~\ref{F:degree_4}}: The figure depicts a partially colored section of a tree with all internal nodes having degree 4.
Alice's and Bob's moves are denoted by odd and even numbers respectively.
Without loss of generality, we can assume that multiple such gadgets are present and Alice's first move was somewhere else in the tree.
Note that \(2,4,6,8 \textnormal{ and } 10\) are leaves and when Bob colors them, their color is forbidden for their parents (\(3,5,7,9 \textnormal{ and } 11\) respectively), thereby forcing Alice to color them.
If Alice does anything else, she loses immediately.
For instance, suppose Alice does not color 3 after Bob colors 2.
Then Bob, in the next move, colors another leaf sibling of 2 with blue and thus both red and blue become forbidden for 3 and Alice loses the game.
Hence Alice is forced to color \(3,5,7,9 \textnormal{ and } 11\).
Then Bob colors \(12\) with red.

Now consider \(v_c\).
Degree of \(v_c\) is 4 and hence, as per the majority condition, at most two neighbors of \(v_c\) can have the same color as \(v_c\).
Three neighbors of \(v_c\) are blue and thus blue is forbidden for \(v_c\).
Degree of vertex \(12\) is also 4 and two of its neighbors are colored same as itself.
Thus red is also forbidden for \(v_c\) because a red on \(v_c\) breaks the majority condition for \(12\).
This shows that Alice needs at least 3 colors to win the majority coloring game.

Now consider any infinite acyclic graph \(G\) for which \(2m \le \deg(v) \le 2m+1\) where \(m\ge 2\).
Since \(\half{2m}=m=\half{2m+1}\), any vertex can have at most \(m\) neighbors with the same color as itself in the majority coloring game.
Hence it reduces to an \(m\)-relaxed coloring game with \(m \ge 2\).
Since \(\chi^m_g(F) \le 2\) for \(m \ge 2\) for any forest \(F\), it implies that \(\mu_g(G)\le2\) for such graphs \(G\).
In other words, Alice has a winning strategy with two colors in the majority coloring game on such graphs.
Hence we make the following remark:
\begin{remark}\label{T:infinite_acyclic_graphs}
    Let \(G\) be an infinite acyclic graph with \(2m \le \deg(v) \le 2m+1 \hspace{0.2cm} \forall v \in V(G) \) where \(m\ge 2\).
    Then, \(\mu_g(G) \le 2\).
\end{remark}

However, the result does not extend to finite acyclic graphs with all internal vertices having degree \(2m\) or \(2m+1\) with a fixed \(m \ge 2\).
The same example shown in figure~\ref{F:degree_4} works for this case also.
Just change the degree of all internal nodes to \(2m\), and make Alice color \(m+1\) children of \(v_c\) blue and \(m\) children of \(12\) red.
At the same time, the presence of leaves does not guarantee that the game majority chromatic number will be larger than the relaxed game chromatic number for that acyclic graph: simply consider the case of a \(3\)-internal-regular tree \(T\). For such a tree of sufficiently large order, we have \(\mu_g(T) = 3 = \chi_g^1(T)\).

\section{\(2\)-caterpillars are game majority \(3\)-colorable}\label{S:Caterpillars}

\subsection{\(1\)-Caterpillars}
\label{S:1caterpillars}
We illustrate the main structural idea on \(1\)-caterpillars, i.e., when every leg of the caterpillar is a leaf, since the strategy on \(2\)-caterpillars is a minor modification to the strategy on \(1\)-caterpillars.
The key reduction here is that since all the legs of the caterpillar have small length, Alice can avoid configuration 2 completely, by the following strategy.

Let \(\mathrm{cat}(k_1,k_2,\dotsc,k_n)\) denote a \(1\)-caterpillar with \(n\) internal nodes having \(k_i\) legs of length \(1\) at the \(i\)th node.

Alice labels the internal nodes and orients the caterpillar such that \(k_1\) is the leftmost node.

In her first move, Alice colors \(k_1\) with any legal color.
For the subsequent moves, Alice colors as follows:
\begin{itemize}
    \item If Bob colors any leg, color the corresponding internal node with a legal color, minimally among its non leg neighbors.
    \item If Bob colors any \(k_i\), then color first uncolored internal node to the right, i.e. color a \(k_j\) such that \(i<j\) and \(k_n\) are colored for \(i<n<j\).
    If no such \(k_i\) is there, color the first uncolored internal node to the right, i.e. \(k_j\) with \(j<i\) such that all \(k_n\) are colored for \(j<n<i\).
    \item Whenever Alice colors any \(k_i\), she uses a color different from that of \(k_{i+1}\).
    If \(k_i+1\) is not colored, she picks a color different from that of \(k_{i-1}\).
    If both the non-leg neighbors are uncolored, she can pick any color.
    \item If there is no uncolored internal node, color any leg.
\end{itemize}

The crucial lemma is the following:

\begin{lemma}\label{L:crit_config_cat}
    For any caterpillar, Alice only needs to take care of \configuration{1}, \configuration{2} is redundant for caterpillars.
\end{lemma}
\begin{proof}
	Let \(m_i := \lfloor \deg(k_i)/2 \rfloor\).
    For any internal node \(k_i\), at most \(2\) of its neighbors are non-leg.
    If the degree of such a \(k_i\) is \(2\) or \(3\), then \(m_i=1\) and one color may be forbidden by giving the same color to \(k_{i-1}\) and \(k_{i+1}\).
    But then we have at most one leg neighbor left (when degree is \(3\)), which can be used to forbid only one color.
    Thus \configuration{2} shown in Figure~\ref{F:config2} cannot arise.
    
    If the degree of any \(k_i\) is more than \(3\), then \(m_i \ge 2\), and coloring \(m+1\) neighbors with the same color will inadvertently involve coloring leg neighbors.
    But a colored leg is itself majority colored, thus for such \(k_i\), \configuration{2} is redundant and reduces to \configuration{1}.

    Similarly, the result holds for \(k_1\) and \(k_n\) also and hence for all \(k_i\).
    Moreover, no critical configuration can arise on the legs.
\end{proof}

\begin{lemma}\label{L:maj_col_cat}
    Any \(k_i\) with \(1<i<n\), in a caterpillar \(C\), cannot be majority colored and have an uncolored non-leg neighbor at the same time if it has more than 1 legs.
\end{lemma}
\begin{proof}
    Consider any \(k_i\) with \(1<i<n\).
    Suppose it has more than 1 leg attached it.
    Since \(1<i<n\), \(k_i\) has two non-leg neighbors.
    Thus, \(\deg(k_i) \ge 4\), and thus \(m_i \ge 2\).
    If degree is greater than 5, i.e., if it has more than 3 legs, than \(k_i\) can never be majority colored, as \(m_i > 2\), but the only neighbors of \(k_i\) which can get the same color as \(k_i\), to make it majority colored, are its two non-leg neighbors.
    If degree is 4 or 5, then \(m_i=2\) and both the neighbors have to be colored but then \(k_i\) cannot have an uncolored non-leg neighbor while being majority colored.
    Hence the result.
\end{proof}

\begin{lemma}\label{L:one_col_leg}
    Following our strategy, Alice ensures that any uncolored \(k_i\) can have at most one colored leg.
\end{lemma}
\begin{proof}
    By our strategy, Alice colors any \(k_i\) immediately as soon as any leg associated to it is colored by Bob.
    Also, she herself does not color legs until all \(k_i\)'s are colored.
    Hence, any uncolored \(k_i\) cannot have more than one colored leg attached to it.
\end{proof}

\begin{lemma}\label{L:diff_col_from_right}
    Our strategy ensures that Alice can always color any \(k_i\) with a color different from that of \(k_{i+1}\).
\end{lemma}
\begin{proof}
    Suppose for contradiction that the result is not true and there is some caterpillar such that Alice has to color an internal node \(k_i\) but the only valid color is that of \(k_{i+1}\).
    Let this be the first instance when such a situation arises.
    
    Now this requires that the other 2 colors are forbidden for \(k_i\).
    By lemma~\ref{L:one_col_leg}, \(k_i\) cannot have two colored legs.
    Thus only one color can be forbidden at leg.
    The other color can be forbidden only if \(k_{i-1}\) is majority colored.
    
\begin{figure}[h!]   
    \centering
    \begin{tikzpicture}[
            every node/.style={circle, draw, fill=white, inner sep=1pt},
            line width=0.9pt,
            smooth
        ]

        \node (0) at (0,0) {0};
        \node (1) at (1,0) {1};
        \node (2) at (2,0) [fill=red!40] {2};
        \node (3) at (3,0) [fill=red!40] {3};
        \node (4) at (4,0) {4};
        \node (7) at (4,-1) [fill=green!40] {7};
        \node (5) at (5,0) [fill=cyan!40] {5};
        \node (6) at (6,0) {6};
        
        \draw (0)--(1)--(2)--(3)--(4)--(5)--(6);
        \draw (4)--(7);
    \end{tikzpicture}
    \caption{Sample Caterpillar (It may have more legs.)}
    \label{F:diff_col_from_right}
\end{figure}
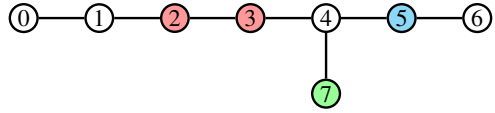

    Thus the configurations looks like the partially colored section of caterpillar shown in figure~\ref{F:diff_col_from_right}.
    Note that by lemma~\ref{L:maj_col_cat}, to be majority colored, the vertex 3 can have at most one leg.

    Alice does not color 7.
    If Bob colors it before 2 and 3 are colored, then Alice will color 4 and we are done.
    So Bob has to color 7 after 2, 3 and 4 are colored.

    Consider 3.
    If Bob colored 3, then Alice would have colored 4 and since 7 was uncolored, she had a legal color different from both 3 and 5.
    Thus 3 should have been colored by Alice.
    Now consider 2.
    Again, if Bob colored 2, Alice would have colored 4 by strategy, using green, which is a color different from 5 and is available as 7 is not colored yet.
    Thus 2 was also colored by Alice.

    But then Alice would have colored 2 with a color different from that of 3, which is its right neighbor and the configuration does not arise.
    This is possible because 4 was the first instance when she could not colored a vertex differently from the right neighbor.

    Hence our assumption that such a configuration arises in some caterpillar was false and the result holds.
    Thus, Alice can always color a vertex differently from its right non-leg neighbor.
\end{proof}

Using the above lemmas, one can argue that any critical configuration that does arise at a vertex \(k_i\) by Alice's strategy must have arisen as a consequence of another critical configuration to its left (or right).
By infinite descent, no critical configuration could have arisen on the caterpillar at all.

\begin{theorem}
    For any caterpillar \(C\), \(\mu_g(C)\le 3\).
\end{theorem}
\begin{proof}
    Suppose the result is not true.
    Then there exists some caterpillar \(C\) which has an uncolored internal node for which no legal color is available.
    By lemma~\ref{L:crit_config_cat}, there has to be \configuration{1} on some internal node \(k_i\), i.e. there is some \(k_i\) such that it has 3 majority colored neighbors, all of different colors.

    By lemma~\ref{L:one_col_leg}, only one majority colored neighbor out of the required 3 can be a leg, the other two have to be \(k_{i-1}\) and \(k_{i+1}\).
    By lemma~\ref{L:maj_col_cat}, \(k_{i-1}\) and \(k_{i+1}\) can have at most one pendant.
    Further, for \(k_{i-1}\) and \(k_{i+1}\) to be majority colored, they need to have one more non-leg neighbor other than \(k_i\).

\begin{figure}[h]
    \centering
    \begin{tikzpicture}[
            every node/.style={circle, draw, fill=white, inner sep=1pt},
            line width=0.9pt,
            smooth
        ]

        \node (0) at (0,0) {0};
        \node (1) at (1,0) {1};
        \node (2) at (2,0) [fill=red!40] {2};
        \node (3) at (3,0) [fill=red!40] {3};
        \node (4) at (4,0) {4};
        \node (9) at (4,-1) [fill=green!40] {9};
        \node (5) at (5,0) [fill=cyan!40] {5};
        \node (6) at (6,0) [fill=cyan!40] {6};
        \node (7) at (7,0) {7};
        \node (8) at (8,0) {8};

        \draw (0)--(1)--(2)--(3)--(4)--(5)--(6)--(7)--(8);
        \draw (4)--(9);
    \end{tikzpicture}
    \caption{The only configuration on which Alice loses on a Caterpillar}
    \label{F:sample_cat}
\end{figure}
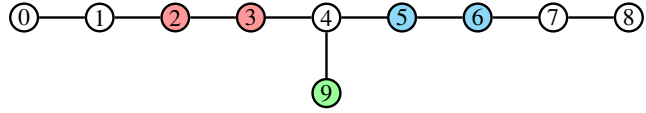

    Consider the caterpillar shown in figure~\ref{F:sample_cat}, no color is legal for node 4.
    Suppose this configuration arises in \move{n}.
    Bob should have colored 9 in the last move.
    If not, then this means 9 was colored by Bob before forbidding either red or blue, in both the cases, Alice following our strategy would have colored 4 instantly and the configuration would not have arisen. Thus 9 was the last vertex in this arrangement to get colored.

    Consider the vertex 3.
    If Bob colored 3, Alice would have colored 4 by our strategy.
    So Alice should have colored 3. This is possible by our strategy, if Bob colored the leg of 3 with a color different from red.
    When Alice colored 3, if 2 was already colored red, then Alice would have chosen the third color, different from red and the color of leg.
    This is possible because 4 is uncolored.
    But then the configuration itself does not arise.
    Thus, 3 is colored red before 2.

    Now consider 2.
    If Bob colors 2 with red, Alice following our strategy will color 4 and since 9 is not colored yet, she has at least one legal color, green in this case and the configuration does not arise.
    Thus, for the configuration to arise, we require Alice to color 2 red.
    But by lemma~\ref{L:diff_col_from_right}, she can color 2 with a color different from red.
    Thus the configuration never arises.
    
\end{proof}

\subsection{\(2\)-Caterpillars}
\label{S:Modified_Caterpillars}
Now, we consider Caterpillars which have leg lengths upto two.
Before starting, we make some important observations.
If a leg has length 2, then the non pendant vertex on that leg has degree two and at most two colors can be forbidden there.
A legal color is always available there.
Likewise, a legal color is always available on the pendant vertices.
Thus Alice only needs to make sure that Bob is not able to forbid 3 colors at any vertex on the spine.
Also, a non pendant vertex on a leg cannot be majority colored if its neighbor on the spine is uncolored.

Now consider any critical configuration on a vertex \(x\) on the spine.
For \configuration{1}, it should have three majority colored neighbors, which is possible if its both neighbors on the spine are majority colored with 2 different colors and it has a pendant vertex attached to it which is colored with the 3rd color.
For \configuration{2}, again by observation, both of its majority colored neighbors have to be its neighbors on the spine and the remaining \(m+1\) neighbors have to be non pendant neighbors on its legs.
This requires \(x\) to have at least \(m+1\) two length legs.
Here, \(m=\half{deg(x)}\).

\begin{theorem}
    If \(C'\) is a caterpillar with leg length upto two, then \(\mu_g(C')\le 3\).
    In other words, Alice has a winning strategy with 3 colors on caterpillars which have leg length upto two.
\end{theorem}
\begin{proof}
    The basic strategy for caterpillars with leg length upto two remains the same as the strategy for caterpillars proposed earlier.
    The only additional thing is that Alice does not color the 2 degree of vertices on any 2 length legs unless all the vertices on the spine are colored and if Bob colors any such 2 degree vertices, she follows the same approach for coloring as she uses when Bob colors any pendant vertices on the legs.

    Now we show that Alice wins using this strategy with three colors.
    As observed earlier, it is sufficient to show that a critical vertex on spine does not appear.

    Consider any critical configuration on the spine at a vertex \(x\).
    Note that by strategy, Alice never colors any vertex on any leg unless all the vertices on the spine are colored.

    In case of \configuration{1}, this means that the leaf neighbor attached to \(x\) has to be colored by Bob.
    If he colors it before the two spine neighbors of \(x\) become colored, then Alice would have colored \(x\) and she had a legal color.
    Thus it was colored after the the two spine neighbors of \(x\) became majority colored.
    Thus this case reduces to the case shown in Figure~\ref{F:sample_cat} and the arguments follow verbatim as in the proof for caterpillars.

    In case of \configuration{2}, Alice would not have colored anything on the legs attached to \(x\).
    As soon as Bob colored the first of those \(m+1\) vertices, Alice would have colored \(x\) by strategy and thus this configuration also cannot arise.

    Hence Alice has a winning strategy with 3 colors on caterpillars which have leg lengths upto two, i.e. \(\mu_g(C')\le 3\).
\end{proof}

\section{Strong majority colorings of cycles}
\label{S:Strong_majority_coloring}

We initiate the study of the strong majority coloring game by giving a complete characterization of the strong majority game chromatic number for cycles.
Recall that the strong majority vertex coloring was introduced by Kalinowski--Kamyczura--Pil{\'s}niak--Wo{\'z}niak in~\cite{KalinowskiEtal2026}.
and in a strong majority coloring, at most half of the neighbors of any given vertex \(v\) can have a given color from the palette, regardless of the color of \(v\).

We consider the following two-person maker-breaker game variation of the aforesaid coloring.
Two players Alice and Bob play the \emph{strong majority coloring game} on a graph \(G\) with a palette of \(k\) colors.
At any point in game, they both have to respect the strong majority condition, i.e. any vertex, whether colored or uncolored, should not have more than half of its neighbors colored using the same color.
Define the \emph{strong majority game chromatic number} as the minimum number of colors required by Alice to win the \emph{strong majority coloring game}.
For a graph \(G\), denote by \(\Maj_g(G)\) its strong majority game chromatic number.

Now we give the complete characterization of the strong majority coloring game for cycles.
For the following discussion, by \(C_n\) we mean a cycle of length \(n\).
We label the vertices of the cycle with \(1 \textnormal{ to } n\) sequentially, starting with \(1\) on the vertex colored by Alice in the first move.

Before starting, we make a key observation which is used repeatedly in the proof:
\begin{claim}
    Both the neighbors of a given vertex in a cycle can not have the same color in a strong majority coloring.
    In other words, in a cycle, vertices separated by distance two necessarily have different colors in a strong majority coloring. 
\end{claim}
\begin{proof}
    In a strong majority coloring, at most half of the neighbors of a given vertex can have any given same color.
    The degree of each vertex in a cycle is \(2\) and \(\half{2}=1\), i.e. at most one neighbor can have any given color.
    Hence the claim.
\end{proof}

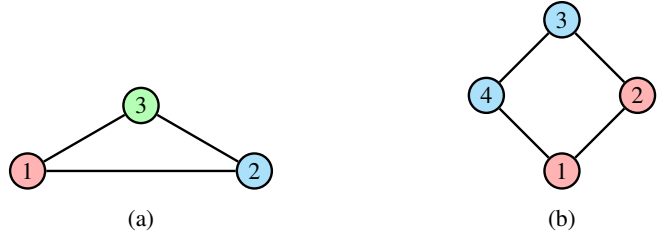
\begin{figure}[h]
    \centering
    \begin{subfigure}[b]{0.45\textwidth}
    \centering
    \begin{tikzpicture}[
            every node/.style={circle, draw, fill=white, inner sep=2pt},
            line width=0.9pt,
            smooth
        ]

        \node (v0) at ( 180:1) [fill=red!30]{1};
        \node (v1) at (   0:2) [fill=cyan!30]{2};
        \node (v2) at (  60:1) [fill=green!30]{3};
        
        \draw (v0)--(v1)--(v2)--(v0);
    \end{tikzpicture}
    \caption{}
    \end{subfigure}
    \begin{subfigure}[b]{0.45\textwidth}
    \centering
    \begin{tikzpicture}[
            every node/.style={circle, draw, fill=white, inner sep=2pt},
            line width=0.9pt,
            smooth
        ]

        \node (v0) at (-90:1) [fill=red!30]{1};
        \node (v1) at (  0:1) [fill=red!30]{2};
        \node (v2) at ( 90:1) [fill=cyan!30]{3};
        \node (v3) at (180:1) [fill=cyan!30]{4};
        
        \draw (v0)--(v1)--(v2)--(v3)--(v0);
    \end{tikzpicture}
    \caption{}
    \end{subfigure}
    \caption{Strong majority coloring of \(C_3\) and \(C_4\)}
    \label{F:strong_mu_c3_c4}
\end{figure}

\begin{theorem}\label{T:strong_mu_cycles}
    For cycles \(C_n\) of length \(n\):
    \[
    \Maj_g(C_n)=\begin{cases}
        2 &n=4\\
        3 &n\ne 4
    \end{cases}
    \]
\end{theorem}
\begin{proof}
    We prove the theorem in different parts:
    \begin{claim}
         \(\Maj_g(C_3)=3\)
    \end{claim}
    As discussed in the claim preceding the theorem, both the neighbors of any give vertex cannot have same color in a strong majority coloring.
    In \(C_3\) any two vertices are neighbors of the remaining vertex.
    If we look at \(1\) and \(2\) as neighbors of \(3\) then they should have different color, say red and blue.
    Now \(3\) is a neighbor of both \(1\) and \(2\).
    So \(3\) cannot get blue as a neighbor of \(1\),i.e. \(2\), is already blue.
    Likewise \(3\) cannot get red as a neighbor of \(2\),i.e. \(1\), is already red.
    And thus Alice loses with two colors.
    In other words, she needs at least 3 colors to win the strong majority coloring game on \(C_3\).
    Since \(C_3\) has only 3 vertices, 3 colors are also sufficient and hence \(\Maj_g(C_3)=3\).

    \begin{claim}
        \(\Maj_g(C_4)=2\)
    \end{claim}
    Suppose Alice and Bob play the strong majority coloring game on \(C_4\) and Alice starts by coloring \(1\) with a color, say red, as shown in figure~\ref{F:strong_mu_c3_c4}.
    As before, \(1\) and \(3\) are neighbors of \(2\) and should have different colors.
    Hence Alice needs at least two colors to win the game.
    Also, \(2\) and \(4\) are neighbors of \(1\) and should have different colors.
    
    This means Bob has following legal moves available:
    \begin{itemize}
        \item Color \(3\) with another color, say blue.
        \item Color a neighbor of \(1\) with blue or red.
    \end{itemize}
    Clearly, no matter what Bob does, he is unable to forbid both colors at a vertex and hence Alice wins with two colors.
    Two colors are therefore required and sufficient and hence \(\Maj_g(C_4)= 2\).

    \begin{claim}
        Alice wins the strong majority coloring game on cycles with 3 colors, i.e. \(\Maj_g(C_n)\le3\).
    \end{claim}
    Consider a cycle \(C_n\) on which Alice and Bob play the strong majority coloring game with \(3\) colors.
    Suppose for contradiction that the theorem is not true.
    Then there is some vertex for which all 3 colors are forbidden.
    Now for any given vertex, a color \(c\) is forbidden if it has a neighbor \(x\) and half of the neighbors of \(x\) are colored \(c\).
    But in a cycle, the degree of all vertices is \(2\) and hence at most \(2\) colors can be forbidden for any given vertex.
    Thus, with a palette of 3 colors, Alice always has a legal color for a given vertex, no matter how she or Bob plays.
    Hence, for cycles, \(\Maj_g(C_n)\le 3\).

    \begin{claim}
        For cycles \(C_n\) with \(n\ge5\), Alice loses the strong majority coloring game with 2 colors, i.e. \(\Maj_g(C_n) \ge 3\) when \(n\ge5\).
    \end{claim}
    Consider a cycle \(C_n\) with length \(n\ge5\).
    Alice and Bob play the strong majority coloring game on \(C_n\) with Alice starting the game by coloring the vertex \(1\).
    Bob in his move colors the \(5^{th}\) vertex with the the other color.
    As discussed in the claim before theorem, vertices which are one vertex apart should have different colors.
    Now consider the \(3^{rd}\) vertex, it is exactly one vertex apart from both \(1\) and \(5\).
    The color of \(3\) should be different from that of \(1\) and \(5\) and as a result no legal color is left for \(3\) as \(1\) and \(5\) have different colors.
    Example of such a play are shown in figure~\ref{F:strong_mu_cycles_ge3}.
    Hence Bob has a winning strategy with \(2\) colors in the strong majority coloring game on cycles of length at least \(5\).
    In other words Alice needs at least 3 colors, i.e. \(\Maj_g(C_n)\ge 3\) when \(n\ge5\).

    Combining this claim with the previous claim, we get \(\Maj_g(C_n)=3\) where \(n\ge 5\). Using this with the previous two claims, we have our theorem:
    \[
    \Maj_g(C_n)=\begin{cases}
        2 &n=4\\
        3 &n\ne 4
    \end{cases}
    \]
\end{proof}

\begin{figure}[h]
    \centering
    \begin{subfigure}[b]{0.25\textwidth}
    \centering
    \begin{tikzpicture}[
            every node/.style={circle, draw, fill=white, inner sep=2pt},
            line width=0.9pt,
            smooth
        ]

        \node (v0) at (-144:1) [fill=red!30]{1};
        \node (v1) at ( -72:1) {2};
        \node (v2) at (   0:1) {3};
        \node (v3) at (  72:1) {4};
        \node (v4) at ( 144:1) [fill=cyan!30]{5};
        
        \draw (v0)--(v1)--(v2)--(v3)--(v4)--(v0);
    \end{tikzpicture}
    \caption{}
    \end{subfigure}
    \begin{subfigure}[b]{0.25\textwidth}
    \centering
    \begin{tikzpicture}[
            every node/.style={circle, draw, fill=white, inner sep=2pt},
            line width=0.9pt,
            smooth
        ]

        \node (v0) at (-120:1) [fill=red!30]{1};
        \node (v1) at ( -60:1) {2};
        \node (v2) at (   0:1) {3};
        \node (v3) at (  60:1) {4};
        \node (v4) at (  120:1) [fill=cyan!30]{5};
        \node (v5) at ( 180:1) {6};
        
        \draw (v0)--(v1)--(v2)--(v3)--(v4)--(v5)--(v0);
    \end{tikzpicture}
    \caption{}
    \end{subfigure}
    \begin{subfigure}[b]{0.4\textwidth}
    \centering
    \begin{tikzpicture}[
            every node/.style={circle, draw, fill=white, inner sep=2pt},
            line width=0.9pt,
            smooth
        ]

        \node (v0) at (-2,0) [fill=red!30]{1};
        \node (v1) at (-1,0) {2};
        \node (v2) at  (0,0) {3};
        \node (v3) at  (1,0) {4};
        \node (v4) at  (2,0)[fill=cyan!30]{5};
        
        \draw (v0)--(v1)--(v2)--(v3)--(v4);
    \end{tikzpicture}
    \caption{}
    \end{subfigure}
    \caption{Strong game majority coloring of cycles}
    \label{F:strong_mu_cycles_ge3}
\end{figure}
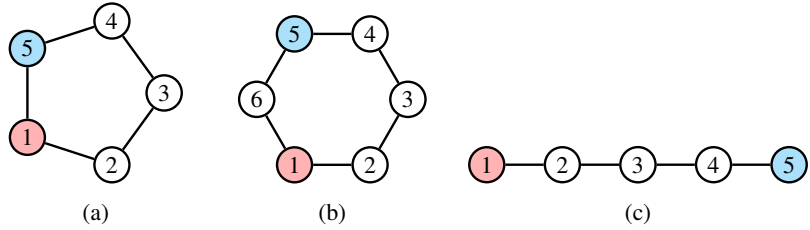

\section{Hardness results related to the strong majority coloring game}

\subsection{Complexity of the decision version of the strong majority coloring game}
\label{S:Complexity}

In this section, we show that it is \PSPACE-complete to determine if Alice has a winning strategy with $k\geq3$ colors for a given graph \(G\) in the Strong Majority Coloring Game.
We formally define the decision version of the problem below.
\defproblem{Strong Majority Game Chromatic Number $(G,n,k)$}{Integer $k\geq3$, graph \(G\) on \(n\) vertices}{Does Alice have a winning strategy on this graph \(G\), using at most $k$ colors?}

The membership of the \textsc{Strong Majority Game Chromatic Number} in \PSPACE can be seen because the maximum number of moves that the game can last for is \(n\) (that is, it is polynomially bounded by the input size), hence, the game tree is of polynomial depth.
Also, in each turn there are at most $k\cdot n$ possible moves. 

Now, we show that the game is \PSPACE-hard by a reduction from the Game Coloring Problem $2$ which was shown to be \PSPACE-hard in \cite{CostaPessoaSampaioSoares2020}.
In this problem, there is a given (uncolored) graph \(G\) and an integer $k$.
In each player's turn, they pick an uncolored vertex and assign it a color which is not assigned to any of its neighbors.
The first player (Alice) wins if the entire graph is colored at the end and the second player (Bob) wins otherwise (i.e., if there is some vertex such that it cannot be assigned any color according to the rule).
This problem is \PSPACE-hard even when $k\geq 3$. We state the decision version of the problem below.

\defproblem{Game Coloring Problem (2)($G,n,m,k$)}{Graph \(G\) on \(n\) vertices and \(m\) edges, integer $k\geq 3$}{Does Alice have a winning strategy with $k$ colors?}

Now we construct an equivalent instance of \textsc{Strong Majority Game Chromatic Number} from a given instance $(G,n,m,k)$ of the \textsc{Game Coloring Problem (2)}.
This new instance is denoted by $(H,n+2m,k)$, where \(n\) is the number of vertices and \(m\) is the number of edges of the given input graph \(G\), $k$ is the number of colors in the given instance of the \textsc{Game Coloring Problem 2}. 

For each vertex $v\in V(G)$, we construct a vertex $v^\prime \in V(H)$.
We refer to $v^\prime$ as copy of \(v\) in \(H\), and the set of vertices in $V(H)$ which are the copy of some vertex in \(V(G)\) as \emph{original vertices}.
Denote the set of all original vertices as $V^\prime$.
For each edge $uv\in E(G)$, add two length-two paths from $u^\prime$ to $v^\prime$ in \(H\). 
Denote the intermediate vertices of these paths by $w^i_{uv}$, for $i\in \{1,2\}$.
That is, for each edge $uv$ in \(G\), there exist two length-two paths $u^\prime w^1_{uv}v^\prime$ and $ u^\prime w^2_{uv}v^\prime$ in \(H\).
Denote the set of vertices $\bigcup_{uv\in E(G)}\{w^1_{uv},w^2_{uv}\}$ by \(W\), and call the vertices in \(W\) as \emph{intermediate vertices}.
The construction of the reduced instance is now complete.
Note that all the original vertices have even degree and hence for any original vertex $v^\prime$, it cannot happen that $\lfloor\frac{v}{2}\rfloor$ of the neighbors of $v$ have one color, $\lfloor\frac{v}{2}\rfloor$ of the neighbors of $v$ have another color, and $v$ still has an uncolored neighbor.
See figure~\ref{fig:construction} for the construction of the reduced instance.

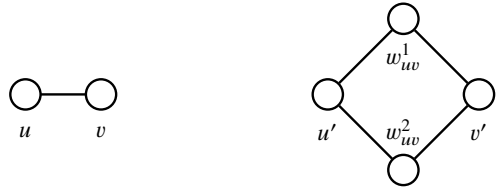
\begin{figure}
    \centering
    \begin{tikzpicture}[
            line width=0.9pt,
            smooth
        ]
    \draw(0,0) circle (0.2 cm);
    \node at (0,-0.5){\(u\)};
    \draw(1,0) circle (0.2 cm);
    \node at (1,-0.5){\(v\)};
    \draw(0.2,0)--(0.8,0);
    \draw(4,0) circle (0.2 cm);
    \node at (4,-0.5){$u^\prime$};
    \draw(6,0) circle (0.2 cm);
    \node at (6,-0.5){$v^\prime$};
    \node at (5,-0.5){$w^2_{uv}$};
    \draw(5,1) circle (0.2 cm);
    \node at (5,0.5){$w^1_{uv}$};
    \draw(5,-1) circle (0.2 cm);
    \draw (4.15,0.15)--(4.85,0.85);
    \draw (5.15,0.85)--(5.85,0.15);
    \draw (4.15,-0.15)--(4.85,-0.85);
    \draw (5.15,-0.85)--(5.85,-0.15);
   \end{tikzpicture}
    \caption{Construction of the two length-two paths in the reduced instance.}
    \label{fig:construction}
\end{figure}

Now we show that Alice has a winning strategy in the Game Coloring Problem (2) on \(G\) if and only if she has a winning strategy in the Strong Majority Coloring Game on \(H\). 

\begin{lemma}\label{lem:fwd}
    If Alice has a winning strategy in the Game Coloring Problem (2) on \(G\), then she has a winning strategy in the Strong Majority Coloring Game on \(H\).
\end{lemma}
\begin{proof}
    We first describe Alice's strategy on \(H\) and then argue that it is indeed a valid winning strategy. 
    \begin{enumerate}
        \item In the first move, suppose Alice colors the vertex \(v\) of \(G\) with color \(i\) according to her winning strategy for the Game Coloring Problem (2) on \(G\), then she colors the copy $v^\prime$ of \(v\) in the Strong Majority Coloring Game on \(H\).
        \item In each subsequent move, if Bob colors an intermediate vertex $w^{i}_{uv}$, with $i\in [1,2]$ with any color; then she responds by coloring the vertex $w^{3-i}_{uv}$ with any legal color, if it is uncolored. If it is already colored, she colors any intermediate vertex with a legal color.
        \item If Bob colors an original vertex say $u^\prime$ with the color \(i\), she assumes that Bob has colored \(u\) in \(G\) with the color \(i\) (such that $u^\prime$ was the copy of \(u\)). Then, according to her winning strategy in the Game Coloring Problem (2) on \(G\), suppose she is required to color the vertex $x$ on \(G\) with the color $j$, then she colors its copy $x^\prime$ on \(H\) with the same color $j$.
        \item If it is Alice's move and there is no original vertex which is uncolored, then she colors an intermediate vertex with a legal color.
    \end{enumerate}

We show that all of these moves are always legal moves for Alice and hence the graph \(H\) is colored at the end of the game, causing Alice to win. Whenever the strategy requires Alice to color an intermediate vertex (say $w^{i}_{uv}$), we need to show that a legal color is always available for her to color it. If a color (say $j$) is not legal for $w^i_{uv}$, then one of the neighbors of $w^{i}_{uv}$ must have exactly half of its neighbors already colored $j$. Since $w^i_{uv}$ has two neighbors \(u\) and \(v\), each of them can cause at most one color to be illegal at $w^i_{uv}$, thus there can be at most two colors which are not legal for $w^i_{uv}$. Since $k\geq3$, there is at least one legal color for Alice to use for coloring $w^i_{uv}$.

Next, we show the following:
\begin{enumerate}
    \item Whenever the strategy requires Alice to color an original vertex $x^\prime$ with a color $j$, it is a legal move in the Strong Majority Coloring Game.
    \item Whenever Bob colors a vertex $u^\prime$ with a legal color $r$ in the Strong Majority Coloring Game on \(H\), it is a legal move in the Game Coloring Problem (2) on \(G\) for Bob to color the vertex \(u\) with the color $r$, thus Alice can safely assume that Bob has colored \(u\) with the color $r$ on \(G\) and proceed.
\end{enumerate} 

Suppose either of the two statements above is not true.
Consider the first instance where one of the two statements is not true.
First, we consider the case where this situation occurs when Alice is required to color the vertex $x^\prime$ with the color $j$, but this move is not legal.
Then, there is a neighbor of $x^\prime$, which has half its neighbors already colored $j$.
By construction, the neighbor of an original vertex is an intermediate vertex, hence there exists a vertex $w^{i}_{xu}$ which has half its neighbors already colored $j$.
The neighbors of $w^{i}_{xu}$ are $x^\prime$ and $u^\prime$ (and recall that $xu$ is an edge in \(G\)), and as $x^\prime$ is uncolored (we assumed that Alice was about to color it but the color $j$ was not legal), the vertex $u^\prime$ must be colored $j$.
According to the Game Coloring Problem (2) on \(G\), the vertex \(u\) must have been colored $j$.
This is because Alice's moves on \(G\) are legal as they are based on a valid winning strategy and Bob's moves on \(G\) are assumed by Alice based on Bob's moves on \(H\), which are legal because till now both the statements $1$ and $2$ are assumed to be true.
Since the vertex \(u\) is colored $j$, coloring the vertex $x$ causes both the endpoints of the edge $xu$ to be colored $j$.
This is a contradiction, as Alice was playing according to a valid strategy on \(G\).  

Now consider the case where statement 2 fails first.
Recall that according to our assumption, both the statements hold true just before this instance.
This means that there is a vertex $u^\prime$ in \(H\) which Bob colors $j$, and this is a legal move in the Strong Majority Coloring Game; but Alice cannot assume that Bob has colored \(u\) in \(G\) with the color $j$, as this is not a legal move for Bob in the Game Coloring Problem (2).
Since the color $j$ is not a legal color for the vertex \(u\) in the Game Coloring Problem (2), this means that \(u\) has a neighbor \(v\) in \(G\), which is colored $j$.
Note that the color of $v^\prime$ is also $j$ in \(H\).
If Bob has colored \(v\), it is a legal move.
Since statement $1$ also holds true before this instance, if Alice colors the vertex $v^\prime$ with the color $j$ in \(H\) (because she colors \(v\) with the color $j$ in \(G\)), it is still a legal move. Now notice that if Bob colors \(u\) with $j$, both the neighbors of the vertex $w^1_{uv}$, i.e., $u^\prime$ and $v^\prime$ would have the color $j$, violating the strong majority condition.
This contradicts the fact that it is legal for Bob to color the vertex $u^\prime$ with the color $j$ in the Strong Majority Coloring Game on \(H\).

Thus, we have shown that Alice indeed has a winning strategy in the Strong Majority Coloring Game on \(H\), if she has a winning strategy in the Game Coloring Problem (2) on \(G\).
\end{proof}

Next, we show that the existence of a winning strategy for Bob on \(G\) in the Game Coloring Problem (2) implies the existence of a winning strategy for Bob on \(H\) in the Strong Majority Coloring Game.

\begin{lemma}\label{lem:rev}
    If Bob has a winning strategy in the Game Coloring Problem (2) on \(G\), then he has a winning strategy in the Majority Coloring Game on \(H\).
\end{lemma}

\begin{proof}
     We first describe Bob's winning strategy on \(H\) and then argue that it is indeed a valid winning strategy. 
    \begin{enumerate}
         \item In each move, if Alice colors an original vertex say $u^\prime$ with the color \(i\), Bob assumes that she has colored \(u\) in \(G\) with the color \(i\) (such that $u^\prime$ was the copy of \(u\)). Then, according to his winning strategy in the Game Coloring Problem (2) on \(G\), suppose he is required to color the vertex $x$ on \(G\) with the color $j$, then he colors its copy $x^\prime$ on \(H\) with the same color $j$.
        \item If it is Bob's move and there is no original vertex which is uncolored, then he colors an intermediate vertex with a legal color.
        \item In each move, if Alice colors an intermediate vertex $w^{i}_{uv}$, with $i\in [1,2]$ with any color; then he responds by coloring the vertex $w^{3-i}_{uv}$ with any legal color, if it is uncolored. Otherwise, he colors any intermediate vertex with a legal color.   
    \end{enumerate}
Now, we show that this strategy always gives a valid move for Bob in the Strong Majority Coloring Game on \(H\), till he wins in the Game Coloring Problem (2) on \(G\).
As soon as he wins the Game Coloring Problem (2) on \(G\), we show that according to this strategy, he also wins the Strong Majority Coloring Game on \(H\). 

For this, note that there is always a legal color available at an intermediate vertex, whenever any player tries to color it.
This is because as shown in the proof of Lemma~\ref{lem:fwd} above, the two neighbors of an intermediate vertex can forbid at most two colors, and as $k\geq 3$, there is always a third color available.

Consider a partial game state where Bob has not yet won in the Game Coloring Problem (2) on \(G\). We show the following:
\begin{enumerate}
    \item Whenever the strategy requires Bob to color an original vertex $x^\prime$ with a color $j$, it is a legal move in the Strong Majority Coloring Game.
    \item Whenever Alice colors a vertex $u^\prime$ with a legal color $r$ in the Strong Majority Coloring Game on \(H\), it is a legal move in the Game Coloring Problem (2) on \(G\) for her to color the vertex \(u\) with the color $r$, thus Bob can safely assume that Alice has colored \(u\) with the color $r$ on \(G\) and proceed.
\end{enumerate}

Consider the first instance where one of the above statements fails.
First, consider the case where statement 1 fails first.
This means that the strategy requires Bob to color an original vertex $x^\prime$ with a color $j$, but it is not a legal move in the Strong Majority Coloring Game.
Then, there is a neighbor of $x^\prime$, which has half its neighbors already colored $j$.
By construction, the neighbor of an original vertex is an intermediate vertex, hence there exists a vertex $w^{i}_{xu}$ which has half its neighbors already colored $j$.
The neighbors of $w^{i}_{xu}$ are $x^\prime$ and $u^\prime$ (and recall that $xu$ is an edge in \(G\)), and as $x^\prime$ is uncolored (we assumed that Bob was about to color it with the color $j$), the vertex $u^\prime$ must be colored $j$.
According to the Game Coloring Problem (2) on \(G\), the vertex \(u\) must have been colored $j$.
This is because Bob has not yet won on \(G\), and both the statements $1$ and $2$ are assumed to be true.
Since the vertex \(u\) is colored $j$, coloring the vertex $x$ causes both the endpoints of the edge $xu$ to be colored $j$.
This is a contradiction, as Bob was playing according to a valid strategy on \(G\), and has not yet won.  

Now consider the case where statement 2 fails first.
This means that it is legal for Alice to color a vertex $u^\prime$ with a color $r$ in the Strong Majority Coloring Game, but it is not legal for her to color the vertex \(u\) with the color $r$ on \(G\).
This means that there exists a neighbor \(v\) of \(u\) in \(G\), such that \(v\) is already colored $r$ in the Game Coloring Problem (2) on \(G\).
Observe that $v^\prime$ must be colored $r$ in the Strong Majority Coloring Game on \(H\).
This is because if \(v\) is colored $r$ on \(G\) by Bob in the Game Coloring Problem (2), then it was legal for him to color $v^\prime$ with the color $r$ in the Strong Majority Coloring Game on \(H\), as he had not yet won on \(G\) and statement 1 is assumed to be true.
If \(v\) is colored $r$ on \(G\) by Alice in the Game Coloring Problem (2), as Bob had not yet won, hence statement 2 holds, and Alice had colored $v^\prime$ with the color $r$ on \(H\) in the Strong Majority Coloring Game, because of which Bob assumed that Alice colored \(v\) with the color $r$ in the Game Coloring Problem (2).
Since $v^\prime$ is colored $r$ on \(H\), the vertex $w^{1}_{uv}$ has one of its two neighbors colored $r$.
Therefore, it is not legal in the Strong Majority Coloring Game for the other neighbor of $w^1_{uv}$, i.e., $u^\prime$ to be colored $r$, which contradicts the fact that it was a legal move for Alice.

It remains to show that as soon as Bob wins in the Game Coloring Problem (2) on \(G\), Bob also wins the Strong Majority Coloring Game on \(H\).
Bob wins in the Game Coloring Problem (2) on \(G\), whenever there is a vertex \(v\) in \(G\) such that it cannot be colored with any of the $k$ colors.
This can only happen when each of the $k$ colors is already present on some neighbor of \(v\).
That is, there is a situation where a vertex \(v\) has neighbors $v_1,v_2,\ldots,v_k$, where $v_1$ is colored with color $1$, $v_2$ is colored with color $2$ and so on.
The vertex \(v\) may have other neighbors. We show that none of the $k$ colors is legal at the vertex \(v\) on \(H\).
Since both statements $1$ and $2$ hold before Bob has won on \(G\), the vertex $v_1^\prime$ is colored with color $1$, the vertex $v_2^\prime$ is colored with color $2$ and so on, in the Strong Majority Coloring Game on \(H\).
Consider a color $i\in [1,k]$. If the color \(i\) is assigned to the vertex $v^\prime$ (by any player), both the neighbors $v^\prime$ and $v_i^\prime$ of the intermediate vertex $w^1_{vv_i}$ would have the color \(i\), which violates the strong majority condition at $w^1_{vv_i}$.
Therefore, none of the $k$ colors is legal at $v^\prime$, and Bob wins the Strong Majority Coloring Game on \(H\).

\end{proof}

The Lemmas~\ref{lem:fwd} and \ref{lem:rev} imply that Alice has a winning strategy in the Game Coloring Problem (2) on \(G\), if and only if, she has a winning strategy in the Strong Majority Coloring Game on \(H\).
This, along with the fact that the reduction takes polynomial time, we get the following result:

\begin{theorem}
The decision version of the \textsc{Strong Majority Game Chromatic Number} is \PSPACE-complete (when $k\geq 3$).
\end{theorem}

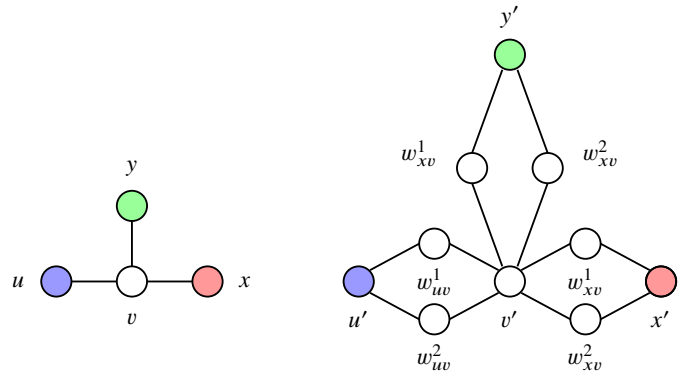
\begin{figure}
    \centering
    \begin{tikzpicture}[smooth,
    line width=0.7pt]
    \draw[fill=blue!40](4,0) circle (0.2 cm);
    \node at (4,-0.5){$u^\prime$};
    \draw(6,0) circle (0.2 cm);
    \node at (6,-0.5){$v^\prime$};
    \draw(5,0.5) circle (0.2 cm);
    \node at (5,0){$w^1_{uv}$};
    \draw(5,-0.5) circle (0.2 cm);
    \node at (5,-1){$w^2_{uv}$};
    \draw (4.15,0.15)--(4.8,0.5);
    \draw (5.2,0.5)--(5.85,0.15);
    \draw (4.15,-0.15)--(4.8,-0.5);
    \draw (5.2,-0.5)--(5.85,-0.15);

     \draw[fill=red!40](8,0) circle (0.2 cm);
    \node at (8,-0.5){$x^\prime$};
    \draw(8,0) circle (0.2 cm);
     \node at (7,0){$w^1_{xv}$};
    \draw(7,-0.5) circle (0.2 cm);
    \node at (7,-1){$w^2_{xv}$};
    \draw(7,0.5) circle (0.2 cm);
    \draw (6.15,0.15)--(6.8,0.5);
    \draw (7.2,0.5)--(7.85,0.15);
    \draw (6.15,-0.15)--(6.8,-0.5);
    \draw (7.2,-0.5)--(7.85,-0.15);

     \draw[fill=green!40](6,3) circle (0.2 cm);
     \node at (6,3.5){$y^\prime$};
      \draw(5.5,1.5) circle (0.2 cm);
     \node at (4.8,1.7){$w^1_{xv}$};
    \draw(6.5,1.5) circle (0.2 cm);
    \node at (7.2,1.7){$w^2_{xv}$};
    \draw (6.1,0.2)--(6.5,1.3);
    \draw (5.9,0.2)--(5.5,1.3);
      \draw (6.1,2.8)--(6.5,1.7);
    \draw (5.5,1.7)--(5.9,2.8);

    \draw[fill=green!40](1,1) circle (0.2 cm);
    \node at (1,1.5){$y$};
    \draw[fill=red!40](2,0) circle (0.2 cm);
    \node at (2.5,0){$x$};
     \draw[fill=blue!40](0,0) circle (0.2 cm);
     \node at (-0.5,0){$u$};
    \draw(1,0) circle (0.2 cm);
    \node at (1,-0.5){$v$};
    \draw(0.2,0)--(0.8,0);
    \draw(1.2,0)--(1.8,0);
    \draw(1,0.2)--(1,0.8);
        
    \end{tikzpicture}
    
    \caption{A subgraph of a graph $G$ in the coloring game with three colors and the corresponding subgraph $H$ in the strong majority coloring game with three colors. Note that Bob wins irrespective of the coloring of the intermediate vertices because the vertex $v^\prime$ cannot be colored blue ($w^1_{uv}$ already has a blue neighbor $u^\prime$), cannot be colored red ($w^1_{xv}$ already has a red neighbor $x^\prime$), and cannot be colored green ($w^1_{yv}$ already has a green neighbor $y^\prime$). In graph $G$, the vertex $v$ cannot be colored blue because it has a blue neighbor $u$, cannot be colored red because it has a red neighbor $x$, and cannot be colored green because it has a green neighbor $y$.}
    \label{fig:strongmajoritycoloringgame}
\end{figure}

\subsection{Complexity of deciding whether a given graph $G$ has a strong majority coloring with $2$ colors}
\label{S:complexity_strong_mu_2color}

The concept of \emph{neighborhood balanced colorings} has been defined in \cite{FreybergMarr2024}, where the authors study $2$-colorings of a graph such that every vertex has exactly half of its neighbors red and the other half blue.
This is exactly the strong majority coloring of a graph with two colors.
However, for more than two colors, these two concepts are not the same.
In case of neighborhood balanced $k$-colorings, a vertex must have exactly \(\frac{1}{k}\) of its neighbors colored with each of the $k$ colors.
Hence, a degree six vertex can have three red neighbors, two blue neighbors and one green neighbor in a strong majority coloring with three colors, however, this is not a valid neighborhood balanced $3$-coloring.
The neighborhood balanced $k$-colorings of graphs have been studied in \cite{AlmeidaSinghEtal2025}.

It was shown in \cite{Asaeedi2024} that it is \NP-complete to determine whether a given graph $G$ (with every vertex having an even degree) admits a neighborhood balanced coloring with $2$-colors.
This would also mean that it is \NP-complete to determine whether a given graph $G$ (with every vertex having an even degree) admits a strong majority coloring with two colors.
The author shows \NP-hardness by a reduction from the set partition problem, which was shown to be \NP-complete by Karp in 1972 \cite{Karp1972}.
We state this problem formally as follows:
 \defproblem{Set Partition}{A set $S=\{a_1,a_2,\ldots,a_k\}$ of $k$ positive integers.}{Does there exist a partition $S_1\uplus S_2$ of $S$ such that the sum of the elements in the two parts is equal, i.e., $\sum_{a\in S_1}a=\sum_{a\in S_2}a$?}

However, this reduction is not polynomial time because for an integer $n$, there are $n$ vertices adjacent to a common vertex drawn in a graph.
But the integer $n$ is given to us as a binary encoding, taking only $log(n)$ bits.
The Set Partition problem is only weakly \NP-hard, i.e., it is \NP-hard when the input integers are given to us as binary encodings, not as unary encodings. 

But, we show that the decision version of the Strong Majority Coloring problem with two colors, is \NP-complete regardless. We state the problem formally below:

\defproblem{Strong Majority $2$-Coloring $G(n)$}{An Eulerian Graph $G$ on $n$ vertices}{Does $G$ have a strong majority coloring with $2$ colors?}

It can be observed that the \textsc{Strong Majority $2$-Coloring} problem is in \NP because given a coloring $c$ it can be efficiently verified whether $c$ is a strong majority coloring.
We show that the Strong Majority $2$-Coloring problem is \NP-hard by a reduction from the Not-All-Equals $3$-SAT problem.

\defproblem{Not All Equals $3$-SAT (NAE $3$-SAT)}{A Boolean formula $\Phi$ in CNF with $m$ clauses and $n$ variables, and each clause containing exactly $3$ literals.}{Does there exist an assignment of each variable $x_i$ to $T/F$ such that $\Phi$ is satisfied, with there being at least one literal assigned $T$, and one literal assigned $F$ in each clause?}

We describe how to construct an equivalent instance of \textsc{Strong Majority $2$-Coloring} in polynomial time.

For each clause $C_j$ ($j\in [1,m]$), create two \emph{clause} vertices $C^1_j$ and $C^2_j$ in the graph $G$, and add an \emph{intermediate} vertex $i_j$ adjacent to both $C^1_j$ and $C^2_j$.
That is, there is a length-two path in $G$ corresponding to each clause $C_j$ ($j\in [1,m]$), with endpoints $C^1_j$ and $C^2_j$.

For each variable $x_i$ ($i\in [1,n]$), create two vertices in $V(G)$, and denote them by $x_i$ and $\bar{x_i}$. We refer to these vertices as \emph{literal vertices}.
Also create two vertices $p_i$ and $q_i$, such that $p_i$ and $q_i$ are both adjacent to $x_i$ and $\bar{x_i}$.
The vertices $p_i,x_i,q_i$ and $\bar{x_i}$ induce a four-cycle with $x_i$ and $\bar{x_i}$ being opposite endpoints, for each $i\in [1,n]$. We refer to the vertices $p_i$ and $q_i$ as \emph{joining vertices}, for each $i\in [1,n]$.

Whenever the literal $x_i$ or $\bar{x_i}$ appears in the clause $C_j$, add edges from the literal vertex $x_i$ or $\bar{x_i}$ to both the clause vertices $C^1_j$ and $C^2_j$.

The description of the construction of the reduced instance is now complete and it is clear that the graph $G$ can be constructed from the formula $\Phi$ in polynomial time.
It can also be verified that the degree of each vertex in the graph $G$ is even.
Now we show that the reduced instance $G(3m+4n)$ of Strong Majority $2$-Coloring is indeed equivalent to the original instance $\Phi$ of NAE $3$-SAT.

\begin{lemma}\label{lem:fw}
If the formula $\Phi$ has a satisfying assignment for the Not-All-Equals $3$-SAT problem, then the graph $G$ has a strong majority coloring with $2$ colors.
\end{lemma}

\begin{proof}
    Consider a satisfying assignment for $\Phi$ in the Not-All-Equals $3$-SAT problem.
    We describe a coloring of the graph $G$ and then show that it is indeed a strong majority coloring.

    \begin{enumerate}
        \item Whenever a variable $x_i$ is set to $T$, color the literal vertex $x_i$ blue, and the literal vertex $\bar{x_i}$ red in $G$.
        \item Whenever a variable $x_i$ is set to $F$, color the literal vertex $x_i$ red, and the literal vertex $\bar{x_i}$ blue in $G$.
        \item For each $i\in [1,n]$, color the joining vertex $p_i$ blue, and the joining vertex $q_i$ red.
        \item For each $j\in [1,m]$, color the clause vertex $C^1_m$ blue, and the clause vertex $C^2_m$ red.
        \item Since this is a satisfying assignment for $\Phi$ in NAE $3$-SAT, no clause has all three literals true or all three literals false.
        Thus, each $C^1_j$ (or $C^2_j$) is adjacent to two blue literal vertices and one red literal vertex, or two red literal vertices and one blue literal vertex.
        For each $j\in [1,m]$, do the following: 
        
        If $C^1_j$ (and $C^2_j$) is adjacent to two blue literal vertices and one red literal vertex, color the intermediate vertex $i_j$ red; otherwise color the intermediate vertex $i_j$ blue.
        
    \end{enumerate}

Each joining vertex $p_i$ (or $q_i$) ($i\in [1,n]$) has one blue neighbor and one red neighbor because as per the coloring, whenever $x_i$ is blue, $\bar{x_i}$ is red, and vice versa.
Each intermediate vertex $i_j$ (for $j\in [1,m]$) has one blue neighbor $C^1_j$ and one red neighbor $C^2_j$.

Consider a literal vertex $x_i$ (or $\bar{x_i}$), where $i\in [1,n]$, such that the literal $x_i$ (or $\bar{x_i}$) belongs to $k$ clauses $C_{j_1}, C_{j_2},\ldots, C_{j_k}$.
The literal vertex $x_i$ (or $\bar{x_i}$) has $k$ blue neighbors $C^1_{j_1}, C^1_{j_2},\ldots, C^1_{j_k}$, and $k$ red neighbors $C^2_{j_1}, C^2_{j_2},\ldots,C^2_{j_k}$, one blue neighbor $p_i$ and one red neighbor $q_i$.
Therefore, each literal vertex $x_i$ or $\bar{x_i}$ has exactly the same number of red and blue neighbors.

Now consider each clause vertex $C^1_j$ (or $C^2_j$), where $j\in [1,m]$.
Since $\Phi$ is a YES instance of NAE $3$-SAT, the clause $C_j$ has two true literals and one false literal, or two false literals and one true literal.
In the first case, the clause vertex $C^1_j$ (or $C^2_j$) has two blue neighbors and one red neighbor among the literal vertices, and the intermediate vertex is colored red in this case.
Thus, the clause vertex $C^1_j$ (or $C^2_j$) has exactly two blue neighbors and two red neighbors.
In the second case, the clause vertex $C^1_j$ (or $C^2_j$) has two red neighbors and one blue neighbor among the literal vertices.
In this case, the intermediate vertex is colored blue. Therefore, in this case also, the clause vertex $C^1_j$ (or $C^2_j$) has exactly two blue neighbors and two red neighbors.

Thus, we have shown that every vertex of $G$ has exactly half of its neighbors colored blue and the other half colored red.
Therefore, this is a strong majority coloring of the graph $G$ with two colors. 

\end{proof}

Now we show that the existence of a strong majority coloring of $G$ with two colors implies the existence of a satisfying assignment for $\Phi$.

\begin{lemma}\label{lem:re}
    If $G$ has a strong majority coloring, then the formula $\Phi$ has a satisfying assignment such that each clause of $\Phi$ contains at least one true literal and one false literal.
\end{lemma}

\begin{proof}
    Suppose that the graph $G$ has a strong majority coloring with two colors.
    Consider a strong majority coloring of the graph $G$ with two colors.
    For any $i\in [1,n]$, the literal vertices $x_i$ and $\bar{x_i}$ cannot have the same color as the vertex $p_i$ has only these two vertices as its neighbors. 
    
    \begin{enumerate}
    \item For each $i\in [1,n]$, whenever the literal vertex $x_i$ is colored blue (and $\bar{x_i}$ is colored red), set the variable $x_i=T$. 
    
    \item For each $i\in [1,n]$, whenever the literal vertex $x_i$ is colored red (and $\bar{x_i}$ is colored blue), set the variable $x_i=F$.
    \end{enumerate}

    In any strong majority coloring of $G$, no clause vertex $C^1_j$ can have three blue or three red neighbors among the literal vertices, because the degree of the vertex $C^1_j$ is four, and it can have at most two neighbors of the same color.
    Therefore, for every $j\in [1,m]$, the clause vertex $C^1_j$ has two blue and one red neighbors, or two red and one blue neighbors among the literal vertices.
    Hence, each clause $C_j$ contains two $T$ and one $F$ literal, or two $F$ and one $T$ literal. 

    Thus, $\Phi$ is a YES instance of the NAE $3$-SAT problem.

\end{proof}
From Lemmas~\ref{lem:fw} and \ref{lem:re}, and the fact that the reduction takes polynomial time, we get the main result below.
\begin{theorem}
    The \textsc{Strong Majority $2$-Coloring} problem in \NP-complete on Eulerian graphs.
\end{theorem}

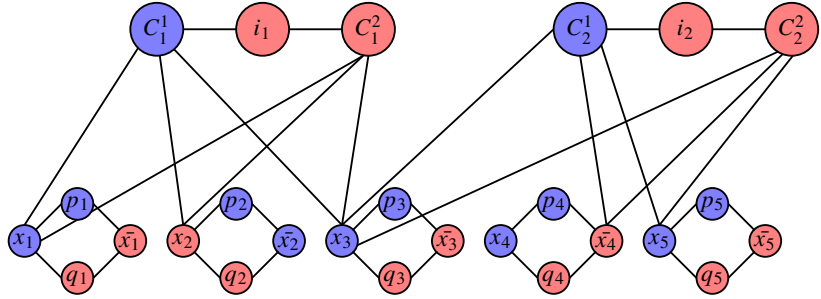
\begin{figure}[h]
    \centering
    \begin{tikzpicture}[scale=0.7,
    line width=0.7,
    smooth]

    \draw(0,0.3)--(2.15,3.65);
    \draw(0.3,0)--(6.4,3.5);
    \draw(3,0.3)--(2.55,3.55);
    \draw(3,0.3)--(6.4,3.5);
    \draw(6,0.3)--(2.8,3.65);
    \draw(6,0.3)--(6.5,3.5);
    \draw(6,0.3)--(10,4);
    \draw(6.25,-0.1)--(14.4,3.65);
    \draw(11,0.3)--(10.5,3.5);
    \draw(11,0.3)--(14.4,3.6);
    \draw(12,0.3)--(10.9,3.7);
    \draw(12,0.3)--(14.6,3.6);
    
   \foreach \x in {0,6,9,12}
         \draw[fill=blue!50](\x,0) circle (0.3 cm);
    \draw[fill=red!50] (3,0) circle (0.3 cm);
    \foreach \x in {0,1,...,4}
      {\pgfmathtruncatemacro{\z}{\x+1}
      \pgfmathtruncatemacro{\y}{3*\x}
      \node at (\y,0){$x_\z$};}
    \foreach \x in {0,6,9,12}
         \draw (\x+2,0) [fill=red!50]circle (0.3 cm);
    \draw(5,0)[fill=blue!50] circle (0.3 cm);
    \foreach \x in {0,1,...,4}
      {\pgfmathtruncatemacro{\z}{\x+1}
      \pgfmathtruncatemacro{\y}{3*\x}
      \node at (\y+2,0){$\bar{x_\z}$};}
    \foreach \x in {0,3,...,12}
        {\draw(\x+0.2,0.2)--(\x+0.7,0.7);
        \draw(\x+0.2,-0.2)--(\x+0.7,-0.7);
        \draw(\x+1.8,0.2)--(\x+1.3,0.7);
        \draw(\x+1.8,-0.2)--(\x+1.3,-0.7);
        }
    \foreach \x in {0,3,...,12}
         {\draw[fill=blue!50] (\x+1,0.7) circle (0.3 cm);
         \draw[fill=red!50](\x+1,-0.7) circle (0.3 cm);
         }
    \foreach \x in {0,1,...,4}
      {\pgfmathtruncatemacro{\z}{\x+1}
      \pgfmathtruncatemacro{\y}{3*\x}
      \node at (\y+1,0.7){$p_\z$};
      \node at (\y+1,-0.7){$q_\z$};}

    \draw[fill=blue!50](2.5,4) circle (0.5 cm);
    \node at (2.5,4){$C^1_1$};
     \draw[fill=red!50](6.5,4) circle (0.5 cm);
    \node at (6.5,4){$C^2_1$};
    \draw(4.5,4)[fill=red!50] circle (0.5 cm);
    \node at (4.5,4){$i_1$};
    \draw(3,4)--(4,4);
    \draw(5,4)--(6,4);

    \draw[fill=blue!50](10.5,4) circle (0.5 cm);
    \node at (10.5,4){$C^1_2$};
    \draw[fill=red!50](14.5,4) circle (0.5 cm);
    \node at (14.5,4){$C^2_2$};
     \draw[fill=red!50](12.5,4) circle (0.5 cm);
    \node at (12.5,4){$i_2$};

    \draw(11,4)--(12,4);
    \draw(13,4)--(14,4);

    \end{tikzpicture}
    \caption{Reduction from NAE $3$-SAT to strong majority coloring with two colors. Consider the instance $\Phi=C_1\land C_2$ of NAE $3$-SAT, where $C_1=x_1\lor x_2\lor x_3$ and $C_2=x_3\lor \bar{x_4}\lor x_5$. The assignment $x_1=x_3=x_4=x_5=T$ and $x_2=F$, is such that each clause contains one true literal and one false literal. We show the corresponding strong majority coloring of the graph $G$ with two colors.}
    \label{fig:strongmajoritytwocolors}
\end{figure}

\section{Concluding Remarks}
\label{S:Conclusion}

In this paper we have showed that \(\mu_g(G)\le3\) for certain classes of graph such as graphs arising from subdivisions, 2-caterpillars, trees with all leaves at a fixed depth \(k\) where \(k\le4\) and certain classes of infinite acyclic graphs.
We gave some natural strategies which failed to be a winning strategy for Alice in the majority coloring game on general trees.
We also initiated the study of the strong majority coloring game by concluding exactly \(\mathrm{Maj}_g(C_n)\) for all \(n\ge3\) and the complexity results including PSPACE-completeness of the \textsc{Strong Majority Game Chromatic Number} problem and the NP-completeness of the \textsc{Strong Majority \(2\)-Coloring} problem on Eulerian graphs.

We close with some interesting questions for further study.
Observe that by the Faigle--Kern--Kierstead--Trotter~\cite{FaigleKernEtAl1993} strategy, if Alice wins the majority coloring game on trees with \(k\) colors, then she also wins with \(k+1\) colors, where \(k \geq 3\).
    So, we ask whether the same holds for \(k = 2\):
\begin{question}
     If Alice wins the majority coloring game on an tree \(T\) with \(2\) colors, then does she also win on \(T\) with \(3\) colors?
\end{question}
An answer to the above question will also answer Zhu's question for the majority coloring game on trees.
Note that Zhu's question is answered positively for trees with maximum degree at most \(4\) by Theorem~\ref{T:FSTTCS}(\ref{T:FSTTCSa}).
Part of the difficulty with the above question is that there does not appear to be a simple description of the trees for which \(\mu_g(T) = 2\).
To illustrate, recall that for nonempty paths we have \(\mu_g(G) =  2\) by Theorem~\ref{T:FSTTCS}(\ref{T:FSTTCSb}).
Suppose we join two paths as shown in figure~\ref{F:Joining_paths}.
The result is a \(1\)-subdivision of the claw \(K_{1,3}\) for which it was shown in Section~\ref{S:Subdivisions} that \(\mu_g(G)=3\).

\begin{figure}
    \centering

    \begin{tikzpicture}[
            every node/.style={circle, draw, fill=white, inner sep=1pt},
            line width=0.9pt,
            smooth
        ]

        \node (0) at (0,0) {0};
        \node (1) at (1,0) {1};
        \node (2) at (2,0) {2};
        \node (3) at (3,0) {3};
        \node (4) at (4,0) {4};
       
        \draw (0)--(1)--(2)--(3)--(4);

        \node at (4,1) [draw=none] {\scalebox{2}{\(+\)}};

        \node (22) at (5,2) {2};
        \node (5) at (5,1) {5};
        \node (6) at (5,0) {6};
        
        \draw (22)--(5)--(6);

        \node at (6,1) [draw=none] {\scalebox{2}{\(=\)}};

        \node (f0) at (7,0) {0};
        \node (f1) at (7,1) {1};
        \node (f2) at (8,2) {2};
        \node (f3) at (9,1) {3};
        \node (f4) at (9,0) {4};
        \node (f5) at (8,1) {5};
        \node (f6) at (8,0) {6};
        
        \draw (f0)--(f1)--(f2)--(f3)--(f4);
        \draw (f2)--(f5)--(f6);
    \end{tikzpicture}

\caption{Joining paths: For the resulting graph \(G\), \(\mu_g(G)=3\)}
\label{F:Joining_paths}
\end{figure}
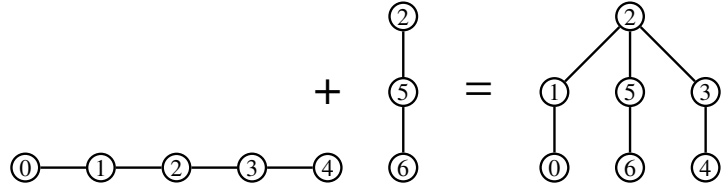

Likewise, consider the two paths joined as shown in Figure~\ref{F:Joining_paths_2}.
The result is a \(2\)-subdivision of the claw \(K_{1,3}\), for which we showed that \(\mu_g(G)=2\) in Section~\ref{S:Subdivisions}.

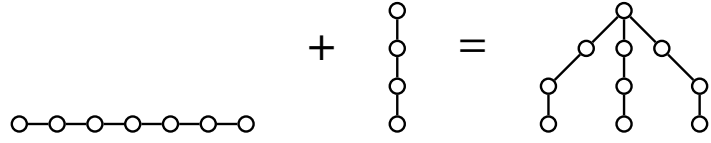
\begin{figure}
    \centering

    \begin{tikzpicture}[
            every node/.style={circle, draw, fill=white, inner sep=2pt},
            line width=0.9pt,
            smooth
        ]

        \node (0) at   (0,0) {};
        \node (1) at  (.5,0) {};
        \node (2) at   (1,0) {};
        \node (3) at (1.5,0) {};
        \node (4) at   (2,0) {};
        \node (5) at (2.5,0) {};
        \node (6) at  (3,0) {};
       
        \draw (0)--(1)--(2)--(3)--(4)--(5)--(6);

        \node at (4,1) [draw=none] {\scalebox{2}{\(+\)}};

        \node (5) at (5,1.5) {};
        \node (6) at   (5,1) {};
        \node (7) at  (5,.5) {};
        \node (8) at   (5,0) {};
        
        \draw (5)--(6)--(7)--(8);

        \node at (6,1) [draw=none] {\scalebox{2}{\(=\)}};

        \node (v0) at (8,1.5) {};
        \node (v1) at (7.5,1) {};
        \node (v2) at ( 8,1) {};
        \node (v3) at (8.5,1) {};
        \node (v4) at (7,.5) {};
        \node (v5) at ( 8,.5) {};
        \node (v6) at ( 9,.5) {};
        \node (v7) at ( 7,0) {};
        \node (v8) at ( 8,0) {};
        \node (v9) at ( 9,0) {};

        \draw (v8)--(v5)--(v2)--(v0)--(v1)--(v4)--(v7);
        \draw (v0)--(v3)--(v6)--(v9);
    \end{tikzpicture}

\caption{Joining paths: For the resulting graph \(G\), \(\mu_g(G)=2\)}
\label{F:Joining_paths_2}
\end{figure}

Moreover, simple degree conditions do not seem to suffice to characterize which trees are game majority \(2\)-chromatic.
Subdivided stars with more than \(3\) legs can also behave like the above subdivisions of the claw \(K_{1,3}\).
More starkly, the proofs of Theorem~\ref{T:FSTTCS}(\ref{T:FSTTCSb}) for paths and stars can be combined to show that if \(T\) is obtained from two arbitrary stars by joining their centers with an arbitrary path, then the resulting tree is game majority \(2\)-chromatic but it can have arbitrarily large degree.
As an aside, we note that the condition on the maximum degree in Theorem~\ref{T:FSTTCS}(\ref{T:FSTTCSa}) can also be relaxed in certain ways: for instance, one can attach arbitrarily many leaves to a leaf-neighbor in a tree \(T\) with \(\Delta(T) \leq 4\), and the same strategy as in the proof of Theorem~\ref{T:FSTTCS}(\ref{T:FSTTCSa}) again works for Alice.

\begin{figure}[h]
    \centering
    \begin{tikzpicture}[every node/.style={circle, draw, fill=white, inner sep=2pt},
    line width=0.8,
    smooth]

    \node (1) at (0,0) {};
    \node (2) at (.5,0) {};
    \node (3) at (1,0) {};
    \node (4) at (1.5,0) {};
    \node (5) at (2,0) {};
    \node (6) at (2.5,0) {};
    \node (7) at (3,0) {};
    \node (8) at (3.5,0) {};
    \node (9) at (4,0) {};
    \node (10) at (4.5,0) {};
    \node (11) at (5,0) {};
    \node (12) at (5.5,0) {};
    \node (13) at (6,0) {};

    \node (l11) at (0,1) {};
    \node (l12) at (0,-1) {};
    \node (l13) at (-.5,-1) {};
    \node (l14) at (-.5,1) {};
    \node (l15) at (-1,-.8) {};
    \node (l16) at (-1,.8) {};
    \node (l17) at (-1.5,-.2) {};
    \node (l18) at (-1.5,.2) {};
    \node (l19) at (-1.4,-.5) {};
    \node (l110) at (-1.4, .5) {};
    \node (l111) at (.5,-1) {};
    \node (l112) at (.5,1) {};

    \node (l21) at (6,1) {};
    \node (l22) at (6,-1) {};
    \node (l23) at (6.5,-1) {};
    \node (l24) at (6.5,1) {};
    \node (l25) at (7,-.8) {};
    \node (l26) at (7,.8) {};
    \node (l27) at (7.5,-.2) {};
    \node (l28) at (7.5,.2) {};
    \node (l29) at (7.4,-.5){};
    \node (l210) at (7.4, .5){};
    \node (l211) at (5.5,-1) {};
    \node (l212) at (5.5,1) {};

    \draw (1)--(2)--(3)--(4)--(5)--(6)--(7)--(8)--(9)--(10)--(11)--(12)--(13);
    \draw (l11)--(1)--(l12);
    \draw (l13)--(1)--(l14);
    \draw (l15)--(1)--(l16);
    \draw (l17)--(1)--(l18);
    \draw (l19)--(1)--(l110);
    \draw (l111)--(1)--(l112);

    \draw (l21)--(13)--(l22);
    \draw (l23)--(13)--(l24);
    \draw (l25)--(13)--(l26);
    \draw (l27)--(13)--(l28);
    \draw (l29)--(13)--(l210);
    \draw (l211)--(13)--(l212);

    \end{tikzpicture}
    \caption{Two high degree stars joined by a path. For the resulting graph \(G\), \(\mu_g(G)=2\).}
    \label{F:stars_joined_path}
\end{figure}
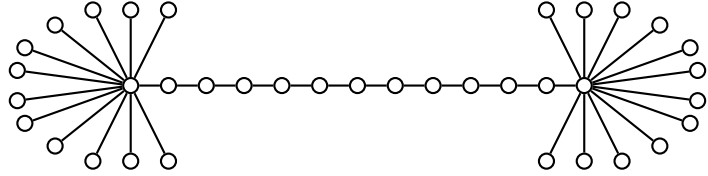

Thus, the caterpillar in Figure~\ref{F:stars_joined_path} above is in fact game majority \(2\)-chromatic.
However, it is not clear whether attaching another path to an internal spine vertex must necessarily increase its game majority chromatic number.
So, as a partial characterization of game majority \(2\)-colorable graphs, and to move towards an answer to the previous question, we ask:
\begin{question}
    Which caterpillars are game majority \(2\)-colorable? 
\end{question}

Bosek--Grytczuk--Jak\'obczak~\cite{BosekGrytczukEtAl2019} mention that it is not clear whether \(\mu_g(T) \leq 3\) for every tree \(T\).
We have discussed some of the difficulties with natural strategies on general trees in Section~\ref{S:Natural_strategy}.
We also found that answering the question for \(3\)-caterpillars was not entirely straightforward.
So, we ask:
\begin{question}
    Is it true that \(\mu_g(C) \leq 3\) for every caterpillar \(C\)?
\end{question}
A positive answer would strengthen the case for a positive answer to the problem raised by Bosek--Grytczuk--Jak\'obczak for a general tree.

Next, we consider some variations of the majority coloring game and restrictions on Alice's strategies:
\begin{question}
    Is there an integer \(m \geq 1\) such that if Alice plays \(m\) moves for every one move of Bob, then she wins the majority coloring game on trees with \(3\) colors?
\end{question}
Of course, if \(m = 1\), then \(\mu_g(T) \leq 3\) for every tree \(T\).

\begin{question}
    For each fixed integer \(r \geq 1\), consider a strategy of Alice on trees which depends only on the \(r\)-neighborhood of the last vertex that Bob colored.
    Does every such strategy of Alice lose on \(3\) colors to Bob on some tree?
\end{question}

From the negative results on natural Alice-strategies, it appears that no ``local'' strategy can win for Alice on arbitrary trees with \(3\) colors. A positive answer to the above question would justify this intuition.

\begin{question}
    Consider the variant of the majority coloring game on trees where Alice uses the palette \(\{1,2,3,4\}\) and Bob can only use the palette \(\{1,2,3\}\).
    Since \(\mu_g(T) \leq 4\) for every tree \(T\), Alice has a winning strategy in this variant on every tree \(T\).
    Does Alice have a winning strategy such that if \(T\) is a tree of order \(n\), then there are \(O(1)\) uses of the color \(4\)?
\end{question}
Of course, if Alice has a winning strategy with zero uses of the color \(4\), then this would prove that \(\mu_g(T) \leq 3\) for every tree \(T\).
We currently do not know of any strategy in which Alice can win with a sublinear number of uses of the color \(4\), so we pose the weaker version as an open question as well:
\begin{question}
    In the setup of the previous question, does Alice have a winning strategy such that if \(T\) is a tree of order \(n\), then there are \(o(n)\) uses of the color \(4\)?
\end{question}

\begin{credits}
\subsubsection{\ackname}
\begin{itemize}
    \item The research of the first author is supported by the University Grants Commission (UGC) Junior Research Fellowship (JRF), Govt.\ of India.
    \item The research of the second author is supported by Anusandhan National Research Foundation (ANRF) National Post-Doctoral Fellowship (N-PDF), Govt.\ of India.
    The second author is also thankful to Arka Ray for helpful discussions about the complexity results.
    \item The second, third, and fourth authors thank the organizers of the Indo-Polish Pre-Conference School on Algorithms and Combinatorics, held in conjunction with CALDAM 2026, where they were introduced to majority colorings.
    \item The authors acknowledge the use of Claude and Gemini for creating the code to implement the \(\lambda\)-strategy and the \((\lambda, \mu)\)-strategy in an interactive format and for assistance in creating the Ti\textit{k}Z code for Figures~\ref{F:game11august} and~\ref{F:game_20_august} from the game screenshots.
    The Python-based code which was used to test the strategy interactively can be found at \url{https://github.com/Yashu112/Majority_coloring_code}.
\end{itemize}

\subsubsection{\discintname}
The authors have no competing interests to declare that are
relevant to the content of this article.
\end{credits}

\bibliographystyle{splncs04}
\bibliography{CALDAM}

\end{document}